\documentclass[preprint,pdftex,english,11pt]{imsart}
\usepackage{amsthm,amsmath,amssymb,amsfonts,epic,multirow,ulem}

\usepackage{thmtools}

\usepackage{tgbonum}

\usepackage[numbers]{natbib}
\RequirePackage{natbib}

\usepackage{wrapfig}
\usepackage[english]{babel}
 \usepackage[T1]{fontenc} 

 \usepackage{makeidx}
 \usepackage{fancyhdr}
\usepackage{epsf}
\usepackage[dvips]{epsfig}  
\usepackage{subfig}
 \usepackage{graphicx}
\usepackage[nice]{nicefrac}
\usepackage{dsfont}
\usepackage{calrsfs} 

\usepackage{pifont}
\usepackage{bm}

\usepackage{wasysym}
\usepackage{stmaryrd}
\usepackage{aeguill}
\usepackage{mathrsfs}
\usepackage{enumerate}
\usepackage[shortlabels]{enumitem}
\usepackage{enumitem}
\usepackage[dvipsnames,table,xcdraw]{xcolor}
\usepackage{cancel}
\usepackage{ulem}
\usepackage{caption}
\usepackage{tikz}

\usepackage{hyperref}
\definecolor{linkcolour}{rgb}{0,0.2,0.6}
\hypersetup{colorlinks,breaklinks,urlcolor=linkcolour, linkcolor=linkcolour, citecolor=linkcolour}

\usepackage{geometry}

\startlocaldefs

\newtheoremstyle{mytheoremstyle} 
        {\topsep}                    
        {\topsep}                    
        {\itshape\fontfamily{ppl}\selectfont}                   
        {}                           
        {\fontfamily{ppl}\selectfont\bfseries\color{black}}                   
        {.}                          
        {.5em}                       
        {}  
\theoremstyle{mytheoremstyle}

\newtheorem{theo}{Theorem}[section]

\newtheorem{prop}[theo]{Proposition}
\newtheorem{coro}[theo]{Corollary}

\newtheorem{lemm}[theo]{Lemma}

\newtheorem{Fact}[theo]{Fact}

\makeatletter
\renewenvironment{proof}[1][\proofname]{\par
  \pushQED{\qed}%
  \fontfamily{ppl} \topsep6\p@\@plus6\p@\relax
  \trivlist
  \item[\hskip\labelsep\itshape\bfseries#1\@addpunct{.}]\ignorespaces}{%
  \popQED\endtrivlist\@endpefalse
}

\def\R{\mathbb{R}}
\def \N{\mathbb{N}}

\def \P{\mathbb{P}} 

\def \E{\mathbb{E}} 

\newcommand{ \un }{\mathds{1}}

\def\T{\mathbb{T}}

\newenvironment{merci}{\textbf{Acknowledgments.}}{ }
\endlocaldefs

\newtheorem{Ass}{Assumption}
\newtheorem{remark}{Remark}

\renewcommand{\T}{\mathbb{T}}
\newcommand{\X}{\mathbb{X}}

\def \Eb{\mathbf{E}}
\def \Pb{\mathbf{P}}

\renewcommand{\P}{\mathbb{P}}

\newtheorem{postita}{Post-it}

\renewcommand{\P}{\mathbb{P}}

\def\T{\mathbb{T}}

\makeatletter
\newcommand*\bigcdot{\mathpalette\bigcdot@{.5}}
\newcommand*\bigcdot@[2]{\mathbin{\vcenter{\hbox{\scalebox{#2}{$\m@th#1\bullet$}}}}}
\makeatother

\newcommand{\VectCoord}[2]{#1^{(#2)}}

\newcommand{\sqrtBis}[1]{#1^{1/2}}

\newcommand{\Pbis}{\mathds{P}}
\newcommand{\Ebis}{\mathds{E}}

\newcommand{\Type}[1]{\mathfrak{t}(#1)}

\begin{document}

{\fontfamily{ppl}\selectfont

\begin{frontmatter}


\title{Scaling limit for local times and return times of a randomly biased walk on a Galton-Watson tree}

\author{\fnms{Alexis} \snm{Kagan}\ead[label=e2]{alexis.l.kagan@amss.ac.cn}}
\address{AMSS, Chinese Academy of Sciences, Beijing, China. \printead{e2}} \vspace{0.5cm}

\runauthor{Kagan}


\runtitle{Scaling limit for return times and local times of a randomly biased walk on a Galton-Watson tree}

\begin{abstract}
We consider a null recurrent random walk $\X$ on a super-critical Galton Watson marked tree $\T$ in the (sub-)diffusive regime. We are interested in the asymptotic behaviour of the local time of its root at $n$, which is the total amount of time spent by the random walk $\X$ on the root of $\T$ up to the time $n$, and in its $n$-th return time to the root of $\T$. We show that properly renormalized, this local time and this $n$-th return time respectively converge in law to the maximum and to an hitting time of some stable Lévy process. This paper aims in particular to extent the results of Y. Hu \cite{Hu2017Corrected}.

\end{abstract}

 \begin{keyword}[class=AenMS]
 \kwd[MSC2020 :  ] {60K37}, {60J80}, {60F17}
 \end{keyword}

\begin{keyword}
\kwd{randomly biased random walks}
\kwd{branching random walks}
\kwd{Galton-Watson trees}
\kwd{local times}
\kwd{return times}
\kwd{scaling limits}
\end{keyword}

\end{frontmatter}


\section{Introduction}

\subsection{Random walk on a Galton-Watson marked tree}\label{RBRWT}
Given, under a probability measure $\Pb$, a $\bigcup_{k\in\N}\R^k$-valued random variable $\mathcal{P}$ ($\mathbb{R}^0$ only contains the sequence with length $0$) with $N:=\#\mathcal{P}$ denoting the cardinal of $\mathcal{P}$, we consider the following Galton-Watson marked tree $(\T,(A_x;x\in\T))$ rooted at $e$: the generation $0$ contains one marked individual $(e,A_e)=(e,0)$. For any $n\in\N^*$, assume the generation $n-1$ has been built. If it is empty, then the generation $n$ is also empty. Otherwise, for any vertex $x$ in the generation $n-1$, let $\mathcal{P}^x:=\{A_{x^{1}},\ldots,A_{x^{N(x)}}\}$ be a random variable distributed as $\mathcal{P}$ where $N(x):=\#\mathcal{P}^x$. The vertex $x$ gives progeny to $N(x)$ marked children $(x^{1},A_{x^{1}}),\ldots,(x^{N(x)},A_{x^{N(x)}})$ independently of the other vertices in generation $n-1$, thus forming the generation $n$. We assume $\Eb[N]>1$ so that $\T$ is a super-critical Galton-Watson tree with offspring $N$, that is $\Pb(\textrm{non-extinction of }\T)>0$ and we define $\Pb^*(\cdot):=\Pb(\cdot|\textrm{non-extinction of }\T)$, where E (resp. $\Eb^*$) denotes the expectation with respect to $\Pb$ (resp. $\Pb^*$) \\
For any vertex $x\in\T$, we denote by $|x|$ the generation of $x$, by $x_i$ its ancestor in generation $i\in\{0,\ldots,|x|\}$ and $x^*:=x_{|x|-1}$ stands for the parent of $x$. In particular, $x_0=e$ and $x_{|x|}=x$. For convenience, we add a parent $e^*$ to the root $e$. For any $x,y\in\T$, we write $x\leq y$ if $x$ is an ancestor of $y$ ($y$ is said to be a descendent of $x$) and $x<y$ if $x\leq y$ and $x\not=y$. We then write $\llbracket x_i,x\rrbracket:=\{x_j; j\in\{i,\ldots,|x|\}\}$. Finally, for any $x,y\in\T$, we denote by $x\land y$ the most recent common ancestor of $x$ and $y$, that is the ancestor $u$ of $x$ and $y$ such that $\max\{|z|;\; z\in\llbracket e,x\rrbracket\cap\llbracket e,y\rrbracket\}=|u|$. \\
Let us now introduce the branching potential $V:x\in\T\mapsto V(x)\in\R$: let $V(e)=A_e=0$ and for any $x\in\T\setminus{\{e}\}$
$$ V(x):=\sum_{e<z\leq x}A_z=\sum_{i=1}^{|x|} A_{x_i}. $$
Under $\Pb$, $\mathcal{E}:=(\T,(V(x);x\in\T))$ is a real valued branching random walk such that $(V(x)-V(x^*))_{|x|=1}$ is distributed as $\mathcal{P}$. We will then refer to $\mathcal{E}$ as the random environment. \\
We are now ready to introduce the main process of our study. Given a realization of the random environment $\mathcal{E}$, we define a $\T\cup\{e^*\}$-valued nearest-neighbour random walk $\X:=(X_j)_{j\in\N}$, reflected in $e^*$ whose transition probabilities are, under the quenched probabilities $\{\P^{\mathcal{E}}_z; z\in\T\cup\{e^*\}\}$: for any $x\in\T$, $\P^{\mathcal{E}}_x(X_0=x)=1$ and 
\begin{align*}
    p^{\mathcal{E}}(x,x^*)=\frac{e^{-V(x)}}{e^{-V(x)}+\sum_{i=1}^{N(x)}e^{-V(x^i)}}\;\;\textrm{ and for all } 1\leq i\leq N(x),\;\; p^{\mathcal{E}}(x,x^i)=\frac{e^{-V(x^i)}}{e^{-V(x)}+\sum_{i=1}^{N(x)}e^{-V(x^i)}}.
\end{align*}
Otherwise, $p^{\mathcal{E}}(x,u)=0$ and $p^{\mathcal{E}}(e^*,e)=1$. Let $\P^{\mathcal{E}}:=\P^{\mathcal{E}}_e$, that is the quenched probability of $\X$ starting from the root $e$ and we finally define the following annealed probabilities
$$ \P(\cdot):=\Eb[\P^{\mathcal{E}}(\cdot)]\;\;\textrm{ and }\;\;\P^*(\cdot):=\Eb^*[\P^{\mathcal{E}}(\cdot)]. $$ 
R. Lyons and R. Pemantle \cite{LyonPema} initiated the study of the randomly biased random walk $\X$. \\
When, for all $x\in\T$, $V(x)=\log\lambda$ for a some constant $\lambda>0$, the walk $\X$ is known as the $\lambda$-biased random walk on $\T\cup\{e^*\}$ and was first introduced by R. Lyons (see \cite{Lyons} and \cite{Lyons2}). The $\lambda$-biased random walk is transient unless the bias is strong enough: if $\lambda\geq\Eb[N]$ then, $\Pb^*$-almost surely, $\X$ is recurrent (positive recurrent if $\lambda>\Eb[N]$). R. Lyons, R. Pemantle and Y. Peres (see \cite{LyonsRussellPemantle1} and \cite{LyonsRussellPemantle2}), later joined by G. Ben Arous, A. Fribergh, N. Gantert, A. Hammond \cite{BA_F_G_H} and E. Aïdékon \cite{AidekonSpeed} for example, studied the transient case and payed a particular attention to the speed $v_{\lambda}:=\lim_{n\to\infty}|X_n|/n\in[0,\infty)$ of the random walk.  \\
When the bias is random, the behavior of $\X$ depends on the fluctuations of the following $\log$-Laplace transform 
\begin{align*}
    \psi(t):=\log\Eb\Big[\sum_{|x|=1}e^{-tV(x)}\Big]=\log\Eb\Big[\sum_{|x|=1}e^{-tA_x}\Big], 
\end{align*}
which we assume to be well defined on $[0,1]$ : as stated by R. Lyons and R. Pemantle \cite{LyonPema}, if $\inf_{t\in[0,1]}\psi(t)$ is positive, then $\Pb^*$-almost surely, $\X$ is transient and we refer to the work of E. Aïdékon \cite{Aidekon2008} for this case. Otherwise, it is recurrent. More specifically, G. Faraud \cite{Faraud} proved that the random walk $\X$ is $\Pb^*$-almost surely positive recurrent either if $\inf_{t\in[0,1]}\psi(t)<0$ or if $\inf_{t\in[0,1]}\psi(t)=0$ and $\psi'(1)>0$. It is null recurrent if $\inf_{t\in[0,1]}\psi(t)=0$ and $\psi'(1)\leq 0$. We refer to the case $\psi'(1)=0$ as the slow random walk on $\T$ since the largest generation reached by the walk $\X$ up to the time $n$ is of order $(\log n)^3$ (see \cite{HuShi10a} and \cite{HuShi10b}).

\subsection{Main results}

In the present paper, we focus on the null recurrent randomly biased walk $\X$ and we put ourselves in the following case for the random environment $\mathcal{E}$:
\begin{Ass}\label{Assumption1}
\begin{align}\label{DiffCase}
    \inf_{t\in[0,1]}\psi(t)=\psi(1)=0\;\;\textrm{ and }\;\;\psi'(1)<0,
\end{align}
\end{Ass}
\noindent and 
\begin{Ass}\label{Assumption2}
    the distribution of the $\bigcup_{k\in\N}\R^k$-valued random variable $\mathcal{P}$ is non-lattice.
\end{Ass}

\noindent Let us introduce
\begin{align}\label{DefKappa}
    \kappa:=\inf\{t>1;\; \psi(t)=0\},
\end{align}
and assume $\kappa\in(1,\infty)$. We make the following technical hypotheses on the random environment $\mathcal{E}$:
\begin{Ass}\label{Assumption3}  for all $\kappa>1$, there exists $\delta_1>0$ such that $\psi(t)<\infty$ for all $t\in(1-\delta_1,\kappa+\delta_1)$ and
\begin{align*}
\left \{\begin{array}{l c l}
    \Eb\Big[\sum_{|x|=1}\max(-V(x),0)e^{-\kappa V(x)}\big]+\Eb\Big[\Big(\sum_{|x|=1}e^{-V(x)}\Big)^{\kappa}\Big]<\infty  & \text{if} & \kappa\in(1,2] \\ \\
    \Eb\Big[\Big(\sum_{|x|=1}e^{-V(x)}\Big)^{2}\Big]<\infty   & \text{if} & \kappa>2.
\end{array}
\right .
\end{align*}
\end{Ass}
\noindent It has been proven that under the assumptions \ref{Assumption1}, \ref{Assumption2} and \ref{Assumption3}, both $|X_n|$ and $\max_{1\leq j\leq n}|X_j|$ are of order $n^{1-1/(\kappa\land 2)}$(see \cite{HuShi10}, \cite{Faraud}, \cite{AidRap} and \cite{deRaph1}). In other words, the random walk $\X$ is sub-diffusive for $\kappa\in(1,2]$ and diffusive for $\kappa>2$. \\
Let $x\in\T\cup\{e^*\}$, $m\in\N^*$ and introduce the local time $\mathcal{L}^m_x$ of $x$ at time $m$:
\begin{align*}
    \mathcal{L}^m_x:=\sum_{j=1}^m\un_{\{X_j=x\}}.
\end{align*}
For convenience, $\mathcal{L}^m$ will stand for the local time $\mathcal{L}^m_{e^*}$ of the parent $e^*$ of the root $e$ at time $m$. \\
The right order for $\mathcal{L}^n$ was first suggested by P. Andreoletti and P. Debs (\cite{AndDeb1}, Proposition 1.3), claiming that for all $\kappa>1$, $\P$-almost surely, $((\log \mathcal{L}^n)/\log n)_{n\in\N^*}$ converges to $1/(\kappa\land 2)$ when $n\to\infty$ and later joined by Y. Hu \cite{Hu2017Corrected} with more precise approximations: let $W_{\infty}$ be the limit of the additive martingale $(\sum_{|x|=\ell}e^{-V(x)})_{\ell\in\N}$ (which is $\Pb^*$-almost surely positive under the assumptions \ref{Assumption1}, \ref{Assumption2} and \ref{Assumption3}) and for $\kappa>2$, introduce the positive constant $c_0:=\Eb[\sum_{x\not=y;\; |x|=|y|=1}e^{-V(x)-V(y)}]/(1-e^{\psi(2)})<\infty$. Then, by (\cite{Hu2017Corrected}, Corollary 1.2), for $\kappa>2$, $\Pb^*$-almost surely, $(\mathcal{L}^n/\sqrtBis{n})_{n\in\N^*}$ converges in law under $\P^{\mathcal{E}}$ to $\sqrtBis{c_0}|\mathcal{N}|/W_{\infty}$, where $\mathcal{N}$ is a standard Gaussian random variable, centered and with variance 1. Moreover, when $\kappa<2$ (resp. $\kappa=2$), $(\mathcal{L}^n/n^{1/\kappa})_{n\in\N^*}$ (resp. $(\mathcal{L}^n/\sqrtBis{(n\log n)})_{n\in\N^*}$) is tight under $\P^{\mathcal{E}}$, for $\Pb^*$-almost every environment. \\
In the present paper, we aim to extend these results. Let us first introduce a few more notations. For all $\kappa\in(1,2]$, let 
\begin{align}\label{TailType1}
    c_{\kappa}:=\lim_{m\to\infty}m^{\kappa}\P\Big(\sum_{x\in\T}
    \un_{\{N_x^{\tau^1}=1,\; \min_{e<z<x}N_y^{\tau^1}\geq 2\}}>m\Big),
\end{align}
with the convention $\sum_{\varnothing}=0$ and for any $x\not=e^*$ and $m>0$, $N_x^m$ is define to be the number of times the oriented edge $(x^*,x)$ has been visited by the random walk $\X$ up to $m$:
\begin{align*}
    N_x^m:=\sum_{j=1}^m\un_{\{X_{j-1}=x^*,\; X_j=x\}},
\end{align*}
and $\tau^j$ stands for the $j$-th hitting time of the oriented edge $(e^*,e)$: $\tau^0=0$ and for any $j\in\N^*$
\begin{align*}
    \tau^j:=\inf\{k>\tau^{j-1};\; X_{k-1}=e^*,\; X_k=e\},
\end{align*}
with the convention $\inf\varnothing=+\infty$. The proof of the existence of $c_{\kappa}\in(0,\infty)$ under the assumptions \ref{Assumption1}, \ref{Assumption2}, \ref{Assumption3} is one of the main purposes of the paper of L. de Raphélis (\cite{deRaph1}, Proposition 2). The equation \eqref{TailType1} is  somewhat reminiscent of the result of Q. Liu \cite{Liu1}, stating that under the assumptions \ref{Assumption1}, \ref{Assumption2} and \ref{Assumption3}, the limit $c_{\infty,\kappa}:=\lim_{m\to\infty}m^{\kappa}\Pb^*(W_{\infty}>m)\in(0,\infty)$ exists. However, there is no reason to believe that $c_{\kappa}$ is equal to $c_{\infty,\kappa}$. \\
For a random process $Y$ and for any $t\geq 0$, $S(t,Y)$ denotes the maximum of $Y$ on the set $[0,t]$:
\begin{align*}
    S(t,Y):=\max_{s\in[0,t]}Y_s.
\end{align*}
Finally, let $(S_j-S_{j-1})_{j\in\N^*}$ be a sequence of\textrm{ i.i.d }real valued random variables such that $S_0=0$ and for any measurable, non-negative function $f:\R\to\R$, $\Eb[f(S_1)]=\Eb[\sum_{|x|=1}e^{-V(x)}f(V(x))]$. This implies the following many-to-one lemma: for any $\ell\in\N^*$ and measurable and non-negative function $\mathfrak{f}:\R^{\ell}\to\R$
\begin{align}\label{ManyTo1}
    \Eb\Big[\sum_{|x|=\ell}e^{-V(x)}\mathfrak{f}\big(V(x_1),\ldots,V(x_{\ell})\big)\Big]=\Eb\big[\mathfrak{f}\big(S_1,\ldots,S_{\ell}\big)\big].
\end{align}
We are now ready to sate our first result:
\begin{theo}\label{Th2}
Assume that the assumptions \ref{Assumption1}, \ref{Assumption2}, \ref{Assumption3} hold. Under $\P^*$ and $\Pb^*$-almost surely, under $\P^{\mathcal{E}}$, the following converges hold in law for the Skorokhod topology on $D([0,\infty))$
\begin{enumerate}[label=(\roman*)]
    \item if $\kappa>2$
    \begin{align*}
        \Big(\frac{1}{\sqrtBis{n}}\mathcal{L}^{\lfloor t n\rfloor};\; t\geq 0\Big)\underset{n\to\infty}{\longrightarrow}\Big(\frac{\sqrtBis{c_0}}{W_{\infty}}S(t,B);\; t\geq 0\Big),
    \end{align*}
    where $B$ is a standard Brownian motion;
    \item if $\kappa=2$
    \begin{align*}
        \Big(\frac{1}{\sqrtBis{(n\log n)}}\mathcal{L}^{\lfloor tn\rfloor};\; t\geq 0\Big)\underset{n\to\infty}{\longrightarrow}\Big(\frac{1}{W_{\infty}}\sqrtBis{(C_{\infty}c_2/2)}S(t,B);\; t\geq 0\Big),
    \end{align*}
    with $C_{\infty}:=\Eb[(\sum_{j\geq 0}e^{-S_j})^{-2}]$, see \eqref{ManyTo1} ;
    \item if $\kappa\in(1,2)$
    \begin{align*}
        \Big(\frac{1}{n^{1/\kappa}}\mathcal{L}^{\lfloor t n\rfloor};\; t\geq 0\Big)\underset{n\to\infty}{\longrightarrow}\Big(\frac{1}{W_{\infty}}(C_{\infty}c_{\kappa}|\Gamma(1-\kappa)|/2)^{1/\kappa}S(t,\VectCoord{Y}{\kappa});\; t\geq 0\Big),
    \end{align*}
    where $\Gamma$ denotes the gamma function and for any $\gamma\in(1,2)$, $\VectCoord{Y}{\gamma}$ is a Lévy process with no positive jump such that for all $\lambda\geq 0$, $\mathtt{E}[e^{\lambda \VectCoord{Y}{\gamma}_t}]=e^{t\lambda^{\gamma}}$. In particular, $\VectCoord{Y}{\gamma}_0=0$ almost surely and $\VectCoord{Y}{\gamma}$ satisfies a scaling property: $\VectCoord{Y}{\gamma}$ and $(\mathfrak{c}^{1/\gamma}\VectCoord{Y}{\gamma}_{t/\mathfrak{c}};\; t\geq 0)$ have the same law for any $\mathfrak{c}>0$. Moreover, $\sup_{t\geq 0}\VectCoord{Y}{\gamma}_t=+\infty$ almost surely.
\end{enumerate}
\end{theo}

\vspace{0.3cm}

\noindent Note that we could have stated Theorem \ref{Th2} for $\mathcal{L}^{\lfloor tn\rfloor}_{e}$ (the local time of the root $e$) instant of $\mathcal{L}^{\lfloor tn\rfloor}$, the proof would have been similar but the random variable $W_{\infty}$ would have been replaced by $p^{\mathcal{E}}(e,e^*)W_{\infty}$, see the remarks below.

\vspace{0.2cm}

\begin{remark}[Remark on $S(t,B)$ and $S(t,\VectCoord{Y}{\gamma})$]\phantom{.}\\
    For convenience, introduce, for any $\gamma\in(1,2]$, $\VectCoord{\Tilde{Y}}{\gamma}=\VectCoord{Y}{\gamma}$ if $\gamma\in(1,2)$ and $\VectCoord{\Tilde{Y}}{\gamma}=B$ if $\gamma=2$. Note that, since, $\VectCoord{\Tilde{Y}}{\gamma}$ has no positive jump, we have that $(S(t,\VectCoord{\Tilde{Y}}{\gamma});\; t\geq 0)$ has continuous paths. Moreover, it satisfies the same scaling property as $\VectCoord{\Tilde{Y}}{\gamma}$. Besides, when $\gamma=2$, $S(1,\VectCoord{\Tilde{Y}}{\gamma})$ is distributed as $|\mathcal{N}|$, where $\mathcal{N}$ is a standard Gaussian random variable. However, when $\gamma\in(1,2)$, $S(1,\VectCoord{\Tilde{Y}}{\gamma})$ is known to follow a Mittag-Leffler distribution of order $1/\gamma$ (see \cite{Bingham1973}): for any $\lambda\geq 0$
    \begin{align}\label{LaplaceMittag-Leffler}
        \mathtt{E}\big[e^{-\lambda S(1,\VectCoord{Y}{\gamma})}\big]=\sum_{k\geq 0}\frac{(-\lambda)^k}{\Gamma(1+\frac{k}{\gamma})}.
    \end{align}
\end{remark}
\noindent Now, introduce $T^j$, the $j$-th hitting time to the parent $e^*$ of the root $e$: $T^0=0$ and for any $j\in\N^*$
\begin{align*}
    T^j:=\{k>T^{j-1};\; X_j=e^*\},
\end{align*}
and for a random process $Y$ and for any $\alpha>0$, $\tau_{\alpha}(Y)$ denotes the first time $Y$ reaches the level $\alpha$ :
\begin{align*}
    \tau_{\alpha}(Y):=\inf\{t>0;\; Y_t=\alpha\}.
\end{align*}
\noindent Similarly to Theorem \ref{Th2}, we have the following result:
\begin{theo}\label{Th1}
Assume that the assumptions \ref{Assumption1}, \ref{Assumption2}, \ref{Assumption3} hold. Under $\P^*$ and $\Pb^*$-almost surely, under $\P^{\mathcal{E}}$, the following converges hold in law for the Skorokhod topology on $D([0,\infty))$
\begin{enumerate}[label=(\roman*)]
    \item if $\kappa>2$
    \begin{align*}
        \Big(\frac{1}{n^2}T^{\lfloor\alpha n\rfloor};\; \alpha\geq 0\Big)\underset{n\to\infty}{\longrightarrow}\Big(\frac{(W_{\infty})^2}{c_0}\tau_{\alpha}(B);\; \alpha\geq 0\Big).
    \end{align*}
    \item if $\kappa=2$
    \begin{align*}
        \Big(\frac{\log n}{n^2}T^{\lfloor\alpha n\rfloor};\; \alpha\geq 0\Big)\underset{n\to\infty}{\longrightarrow}\Big(\frac{(W_{\infty})^2}{C_{\infty}c_2}\tau_{\alpha}(B);\; \alpha\geq 0\Big);
    \end{align*}
    \item if $\kappa\in(1,2)$
    \begin{align*}
        \Big(\frac{1}{n^{\kappa}}T^{\lfloor\alpha n\rfloor};\; \alpha\geq 0\Big)\underset{n\to\infty}{\longrightarrow}\Big(\frac{2(W_{\infty})^{\kappa}}{C_{\infty}c_{\kappa}|\Gamma(1-\kappa)|}\tau_{\alpha}(\VectCoord{Y}{\kappa});\; \alpha\geq 0\Big).
    \end{align*}
\end{enumerate} 
\end{theo}

\begin{remark}[Remark on $\tau_{\alpha}(B)$ and $\tau_{\alpha}(\VectCoord{Y}{\gamma})$]\label{RemTA}\phantom{.}\\
Recall that for any $\gamma\in(1,2]$, $\VectCoord{\Tilde{Y}}{\gamma}=\VectCoord{Y}{\gamma}$ if $\gamma\in(1,2)$ and $\VectCoord{\Tilde{Y}}{\gamma}=B$ if $\gamma=2$. First, note that, since $\sup_{t\geq 0}\VectCoord{\Tilde{Y}}{\gamma}_t=+\infty$, almost surely, we have $\tau_{\alpha}(\VectCoord{\Tilde{Y}}{\gamma})<\infty$ almost surely for all $\alpha>0$. Then, the fact that $\VectCoord{\Tilde{Y}}{\gamma}$ has no positive jump implies that for any $0<\alpha\leq\beta$, $\tau_{\alpha}(\VectCoord{\Tilde{Y}}{\gamma})\leq\tau_{\beta}(\VectCoord{\Tilde{Y}}{\gamma})$. Hence, thanks to the strong Markov property and the scaling property of $\VectCoord{\Tilde{Y}}{\gamma}$, $(\tau_{\alpha}(\VectCoord{\Tilde{Y}}{\gamma});\; \alpha\geq 0)$ is a subordinator thus admitting a càdlàg modification. We use systematically this modification in the present paper. Finally, recall (see \cite{Zolotarev1} for example) that, for all $\lambda\geq 0$ and any $\alpha>0$
\begin{align}\label{LaplaceTA}
    \mathtt{E}\big[e^{-\lambda\tau_{\alpha}(\VectCoord{\Tilde{Y}}{\gamma})}\big]=\left \{
   \begin{array}{l c l}
      e^{-\alpha\lambda^{1/\gamma}}  & \text{if} & \gamma\in(1,2) \\
      e^{-\alpha\sqrtBis{(2\lambda)}}   & \text{if} & \gamma=2. 
   \end{array}
   \right .
\end{align}
Note that thanks to \eqref{LaplaceTA}, $(\tau_{\alpha}(\VectCoord{\Tilde{Y}}{\gamma});\; \alpha>0)$ also satisfies a scaling property: the processes $(\tau_{\alpha}(\VectCoord{\Tilde{Y}}{\gamma});\; \alpha\geq 0)$ and $(\mathfrak{c}^{\gamma\land 2}\tau_{\alpha/\mathfrak{c}}(\VectCoord{\Tilde{Y}}{\gamma});\; \alpha\geq0)$ have the same law.
\end{remark}

\begin{remark}
    Let $\kappa_n=n^{\kappa\land 2}$ if $\kappa\not=2$ and $\kappa_n=n^2/\log n$ if $\kappa=2$. Note that we actually obtain the convergence of the sequence of processes $((\frac{1}{\kappa_n}T^{\lfloor\alpha n\rfloor};\; \alpha\geq 0),(\frac{1}{n}\mathcal{L}^{\lfloor t\kappa_n\rfloor}\; t\geq 0))_{n\in\N^*}$ in law for the Skorokhod product topology on $D([0,\infty))\times D([0,\infty))$.
\end{remark}

\noindent For any $m\in\N^*$, introduce $\mathcal{R}_m$, the range of the random walk $\X$ at time $m$, that is the sub-tree of $\T$ of all the vertices visited by $\X$ up to $m$:
\begin{align*}
    \mathcal{R}_m:=\{x\in\T;\; \exists\; j\leq m:\; X_j=x\},
\end{align*}
and denote by $R_m$ its cardinal. The result we would like to present links $T^j$ and $R_{T^j}$ uniformly:

\begin{theo}\label{Th3}
Assume that the assumptions \ref{Assumption1}, \ref{Assumption2}, \ref{Assumption3} hold. For all $\varepsilon>0$ and any $\kappa>1$
\begin{align*}
    \P^*\Big(\sup_{1\leq p\leq n}\big|R_{T^p}-\frac{\bm{c}_{\infty}}{2}T^p\big|>\varepsilon\kappa_n\Big)\underset{n\to\infty}{\longrightarrow}0,
\end{align*}
where $\bm{c}_{\infty}:=\Eb[(\sum_{j\geq 0}e^{-S_j})^{-1}]$, see \eqref{ManyTo1}.
\end{theo}

\noindent Theorem \ref{Th3} somehow says that the size of the tree $\mathcal{R}_{T^p}$ is proportional to the time spent by the random walk $\X$ away from $e^*$, uniformly in $p$. Note that it is known (see the proofs of Theorem 1.1 and Theorem 6.1 in \cite{AidRap} for $\kappa>2$ and Lemma 4 in \cite{deRaph1} for $\kappa\in(1,2]$) that $R_n/n\to \bm{c}_{\infty}/2$ as $n\to\infty$ in $\P^*$-probability.

\vspace{0.2cm}

\noindent In our last result, we give an equivalent of the probability $\P^{\mathcal{E}}(X_{2n+1}=e^*)$ as $n\to\infty$:
\begin{coro}\label{Coro1}
Assume that the assumptions \ref{Assumption1}, \ref{Assumption2}, \ref{Assumption3} hold. We have, $\Pb^*$-almost surely
\begin{enumerate}[label=(\roman*)]
    \item if $\kappa>2$
        \begin{align*}
            \lim_{n\to\infty}\sqrtBis{n}\P^{\mathcal{E}}(X_{2n+1}=e^*)=\frac{1}{W_{\infty}}\Big(\frac{2c_0}{\pi}\Big)^{1/2}; 
        \end{align*}
    \item if $\kappa=2$
        \begin{align*}
            \lim_{n\to\infty}\Big(\frac{n}{\log n}\Big)^{1/2}\P^{\mathcal{E}}(X_{2n+1}=e^*)=\frac{1}{W_{\infty}}\Big(\frac{C_{\infty}c_2}{\pi}\Big)^{1/2};
        \end{align*}
    \item if $\kappa\in(1,2)$
        \begin{align*}
            \lim_{n\to\infty}n^{1-1/\kappa}\P^{\mathcal{E}}(X_{2n+1}=e^*)=\frac{1}{W_{\infty}}\Gamma(1/\kappa)^{-1}\big(C_{\infty}c_{\kappa}|\Gamma(1-\kappa)|\big)^{1/\kappa}.
        \end{align*}
\end{enumerate}
\end{coro}

\vspace{0.3cm}

\noindent Again, we could have stated Corollary \ref{Coro1} for $\P^{\mathcal{E}}(X_{2n}=e)$ instead of $\P^{\mathcal{E}}(X_{2n+1}=e^*)$ and the random variable $W_{\infty}$ would have been replaced by $p^{\mathcal{E}}(e,e^*)W_{\infty}$.

\vspace{0.2cm}

\noindent Corollary \ref{Coro1} extends Corollary 1.4 in \cite{Hu2017Corrected}. Using concave recursive equations, Y. Hu was able to find an equivalent of $\P^{\mathcal{E}}(X_{2n}=e)$ for any $\kappa>2$ but, unfortunately, these arguments  did not permit to obtain such precision when $\kappa\leq 2$. Indeed, it is proved that $\Pb^*$-almost surely, if $\kappa\in(1,2)$, then $0<\liminf_{n\to\infty}n^{1-1/\kappa}\P^{\mathcal{E}}(X_{2n}=e)\leq\limsup_{n\to\infty}n^{1-1/\kappa}\P^{\mathcal{E}}(X_{2n}=e)<\infty$ and if $\kappa=2$, then, $0<\liminf_{n\to\infty}\sqrtBis{(n/\log n)}\P^{\mathcal{E}}(X_{2n}=e)\leq\limsup_{n\to\infty}\sqrtBis{(n/\log n)}\P^{\mathcal{E}}(X_{2n}=e)<\infty$.

\subsection{Multi-type Galton Watson forests}

We first give a brief construction of a multi-type Galton-Watson tree $(\mathcal{T},(\beta(x);x\in\mathcal{T}))$, rooted at $e$ with initial type $\beta(e)$ such that $\beta(x)\in\N^*$ for all $x\in\mathcal{T}$. Under a probability measure $\Pbis$, let $(\mathcal{P}_i)_{i\in\N^*}$ be a sequence of $\bigcup_{k\in\N}(\N^*)^k$-valued random variables.The generation $0$ contains one individual, the root $e$ with type equal to $\beta(e)$. For any $n\in\N^*$, assume the generation $n-1$ has been built. If it is empty, then the generation $n$ is also empty. Otherwise, for any vertex $x$ in the generation $n-1$, with type equal to $\beta(x)$, let $\mathcal{P}^x:=\{\beta(x^1),\ldots,\beta(x^{N(x)})\}$ be a random variable distributed as $\mathcal{P}_{\beta(x)}$, where $N(x):=\#\mathcal{P}^x$. The vertex $x$ gives progeny to $N(x)$ typed individuals $(x^1,\beta(x^1)),\ldots,(x^{N(x)},\beta(x^{N(x)}))$ independently of other vertices in generation $n-1$, thus forming the generation $n$. A multi-type Galton-Watson forest is then nothing but a family composed of\textrm{ i.i.d }multi-type Galton-Watson trees. \\
In this section, we consider a multi-type Galton-Watson forest $\mathcal{F}:=((\mathcal{T}_i,(\beta(x);x\in\mathcal{T}_i));i\in\N^*)$ satisfying the following hypothesis: if $\VectCoord{e}{i}$ denotes the root of the tree $\mathcal{T}_i$, then
\begin{enumerate}[start=1,label={(\bfseries H\arabic*)}]
    \item\label{H1} $\Pbis$-almost surely, $\forall\; i\in\N^*$, $\beta(\VectCoord{e}{i})=1$;
    \item\label{H2} $\Ebis[\sum_{x\in\mathcal{T}_1}\un_{\{G^1_{1}(x)=1,\; \beta(x)=1\}}]=1$ where $G^1_i(x)$ denotes the number of vertices of type 1 in genealogical line of $x$:
    \begin{align*}
        \forall\;  u\in\mathcal{T}_i\setminus\{\VectCoord{e}{i}\},\; G^1_i(u):=\sum_{\VectCoord{e}{i}\leq z<u}\un_{\{\beta(z)=1\}}\;\;\textrm{ and }\;\; G^1_i(\VectCoord{e}{i})=0;
    \end{align*}
    \item\label{H3} $\Ebis[\sum_{x\in\mathcal{T}_1}\beta(x)\un_{\{G^1_1(x)=1\}}]=:\nu<\infty$.
\end{enumerate}

\noindent Also denote by $\Tilde{\nu}$ the mean of $\sum_{x\in\mathcal{T}_1}\un_{\{G^1_1(x)=1\}}$. For all $p\in\N^*$, let $\mathcal{F}_p:=(\mathcal{T}_i)_{i\in\{1,\ldots,p\}}$ be the family composed of the first $p$ trees of the forest $\mathcal{F}$ and for any $x\in\mathcal{T}_i$, introduce
\begin{align*}
    \beta^{\star}(x):=\beta(x)+\sum_{y\in\mathcal{T}_i;\; y^*=x}\beta(y).
\end{align*}
Let $F_p=\sum_{x\in\mathcal{F}_p}\beta^{\star}(x)$ and $\Bar{F}(m):=\sup\{p\in\N^*;\; F_{p}\leq m\}$. We are interested in the asymptotic behavior of $(F_p,\Bar{F}(m))$ when $p$ and $m$ go to $\infty$. We deal with it in the following proposition:

\begin{prop}\label{PropArbresMulti}
Assume that $\mathcal{F}=((\mathcal{T}_i,(\beta(x);x\in\mathcal{T}_i));i\in\N^*)$ is a multi-type Galton-Watson forest satisfying hypothesis \ref{H1}, \ref{H2} and \ref{H3}. We have
 \begin{enumerate}[label=(\roman*)]
     \item if $\Ebis[(\sum_{x\in\mathcal{T}_1}\un_{\{\beta(x)=1,\; G^1_1(x)=1\}})^2]<\infty$, then, in law, under $\Pbis$, for the Skorokhod product topology on $D([0,\infty))\times D([0,\infty))$ 
     \begin{align*}
        &\Big(\Big(\frac{1}{n}F_{\lfloor\alpha \sqrtBis{n}\rfloor};\; \alpha\geq 0\Big),\Big(\frac{1}{\sqrtBis{n}}\Bar{F}({\lfloor tn\rfloor});\; t\geq 0\Big)\Big) \\[0.7em] & \underset{n\to\infty}{\longrightarrow}\Big(\Big(\frac{2\nu}{\sigma_1^2}\tau_{\alpha}(B);\; \alpha\geq 0\Big),\Big(\Big(\frac{\sigma_1^2}{2\nu}\Big)^{1/2}S(t,B);\; t\geq 0\Big)\Big),
    \end{align*}
    where $\sigma^2_1\in(0,\infty)$ denotes the variance of the random variable $\sum_{x\in\mathcal{T}_1}\un_{\{\beta(x)=1,\; G^1_1(x)=1\}}$. Moreover, for any $M,\varepsilon_1>0$
    \begin{align*}
        \lim_{n\to\infty}\Pbis\Big(\sup_{1\leq p\leq\lfloor M\sqrtBis{n}\rfloor}\Big|\sum_{x\in\mathcal{F}_{p}}1-\frac{\Tilde{\nu}}{2\nu}F_{p}\Big|>\varepsilon_1 n\Big)=0;
    \end{align*}
    \item if $\lim_{m\to\infty}m^2\Pbis(\sum_{x\in\mathcal{T}_1}\un_{\{\beta(x)=1,\; G^1_1(x)=1\}}>m):=\mathfrak{c}_2\in(0,\infty)$ exists, then in law, under $\Pbis$, for the Skorokhod product topology on $D([0,\infty))\times D([0,\infty))$  
    \begin{align*}
        &\Big(\Big(\frac{1}{n}F_{\lfloor\alpha \sqrtBis{(n\log n)}\rfloor};\; \alpha\geq 0\Big),\Big(\frac{1}{\sqrtBis{(n\log n)}}\Bar{F}({\lfloor tn\rfloor});\; t\geq 0\Big)\Big) \\[0.7em] & \underset{n\to\infty}{\longrightarrow}\Big(\Big(\frac{2\nu}{\mathfrak{c}_2}\tau_{\alpha}(B);\; \alpha\geq 0\Big),\Big(\Big(\frac{\mathfrak{c}_2}{2\nu}\Big)^{1/2}S(t,B);\; t\geq 0\Big)\Big).
    \end{align*}
    Moreover, for any $M,\varepsilon_1>0$
    \begin{align*}
        \lim_{n\to\infty}\Pbis\Big(\sup_{1\leq p\leq\sqrtBis{\lfloor M(n\log n)}\rfloor}\Big|\sum_{x\in\mathcal{F}_p}1-\frac{\Tilde{\nu}}{2\nu}F_p\Big|>\varepsilon_1 n\Big)=0;
    \end{align*}
    \item if there exists $\gamma\in(1,2)$ such that $\lim_{m\to\infty}m^{\gamma}\Pbis(\sum_{x\in\mathcal{T}_1}\un_{\{\beta(x)=1,\; G^1_1(x)=1\}}>m):=\mathfrak{c}_{\gamma}\in(0,\infty)$ exists, then in law, under $\Pbis$, for the Skorokhod product topology on $D([0,\infty))\times D([0,\infty))$
    \begin{align*}
        &\Big(\Big(\frac{1}{n}F_{\lfloor\alpha n^{1/\kappa}\rfloor};\; \alpha\geq 0\Big),\Big(\frac{1}{n^{1/\kappa}}\Bar{F}({\lfloor tn\rfloor});\; t\geq 0\Big)\Big) \\[0.7em] & \underset{n\to\infty}{\longrightarrow}\Big(\Big(\frac{2\nu}{\mathfrak{c}_{\gamma}|\Gamma(1-\gamma)|}\tau_{\alpha}(\VectCoord{Y}{\gamma});\; \alpha\geq 0\Big),\Big(\big(\mathfrak{c}_{\gamma}|\Gamma(1-\gamma)|/(2\nu)\big)^{1/\kappa}S(t,\VectCoord{Y}{\gamma});\; t\geq 0\Big)\Big).
    \end{align*}
    Moreover, for any $M,\varepsilon_1>0$
    \begin{align*}
        \lim_{n\to\infty}\Pbis\Big(\sup_{1\leq p\leq\lfloor Mn^{1/\kappa}\rfloor}\Big|\sum_{x\in\mathcal{F}_p}1-\frac{\Tilde{\nu}}{2\nu}F_p\Big|>\varepsilon_1 n\Big)=0.
    \end{align*}
 \end{enumerate}
\end{prop}

\section{Proofs of our main results}

This section is devoted to the proofs of our three theorems and our corollary, it is organized as follows: we first give several preliminary results on the random walk $\X$ that will be useful for the proofs of our theorems. With these preliminary results, we will be able to show Theorem \ref{Th2}, Theorem \ref{Th1} and Theorem \ref{Th3}. We end this section with the proof of Corollary \ref{Coro1}.

\subsection{Preliminary results}

The proof of Theorem \ref{Th2} and the ones of Theorem \ref{Th1} and \ref{Th3} share several common preliminary steps we collect in this section.

\vspace{0.1cm}

\noindent Let us recall that $\tau^0=0$, $\tau^j=\inf\{k>\tau^{j-1}; X_{k-1}=e^*,X_k=e\}$ for any $j\in\N^*$, for all $m\in\N^*$ and any $x\in\T$, $N_x^m=\sum_{j=1}^m\un_{\{X_{j-1}=x^*, X_j=x\}}$ and $\mathcal{R}_m=\{x\in\T;\;N_x^m\geq 1\}$.
\begin{Fact}[Lemma 3.1, \cite{AidRap}]\label{FactGWMulti}
    Under $\P$, for any $p\in\N^*$, $(\mathcal{R}_{\tau^p},(N_x^{\tau^p};\; x\in\mathcal{R}_{\tau^p}))$ is a multi-type Galton-Watson tree with initial type equal to $p$. Moreover, we have the following characterization: for any $x\in\T$, $x\not=e$ and any $k_1,\ldots,k_{N(x)}\in\N^*$
    \begin{align*}
        \P^{\mathcal{E}}\Big(\bigcap_{i=1}^{N(x)}\{N_{x^i}^{\tau^1}=k_i\}\big|N_z^{\tau^1};\; z\leq x\Big)=\frac{(N_x^{\tau^1}-1+\sum_{i=1}^{N(x)}k_i)!}{(N_x^{\tau^1}-1)!k_1!\cdots k_{N(x)}!}\times p^{\mathcal{E}}(x,x^*)^{N_x^{\tau^1}}\prod_{i=1}^{N(x)}p^{\mathcal{E}}(x,x^i)^{k_i}.
    \end{align*}
 \end{Fact}   
\noindent Fact \ref{FactGWMulti}, together with the independence of the increments of the branching random walk $(\T,(V(x);\; x\in\T))$ leads to the following mixed branching property:

\vspace{0.3cm}

\noindent \textbf{The mixed branching property} 

\vspace{0.1cm}

\noindent Let $p,\ell\in\N^*$ and denote by $\mathcal{G}^m_{\ell}$ the sigma-algebra generated by the collection $\{x\in\T,\; |x|\leq\ell;\; N_x^{m}, V(x)\}$. Let $z\in\bigcup_{k\in\N^*}\N^k$. For any measurable and non-negative function $f_z:\N^{\N}\times\R^{\N}\to\R$
\begin{align}\label{BP}
    \E\Big[\prod_{|u|=\ell}\un_{\{N_u^{\tau^p}\geq 1\}}f_u\big((N_x^{\tau^p};\; x>u),(V_u(x);\; x>u)\big)\big|\mathcal{G}^{\tau^p}_{\ell}\Big]=\prod_{|u|=\ell}\un_{\{N_u^{\tau^p}\geq 1\}}F_u\big(N_u^{\tau^p}\big),
    \tag{\textbf{MBP}}
\end{align}
where for any $m\in\N^*$, $F_u(m):=\E[f_u((N_x^{\tau^m}; x>e),(V(x); x>e))]$. We will always refer to \eqref{BP} as the mixed branching property.
\begin{remark}[Return times to the root $e$]
    Define $T_e^j$ to be the $j$-th return time to the root $e$: $T_e=0$ and for any $j\in\N^*$, $T^j_e=\inf\{k>T^{j-1}_e;\; X_k=e\}$. Note that, under $\P$, $(\mathcal{R}_{T_e^p},(N_x^{T_e^p};\; x\in\mathcal{R}_{T_e^p}))$ is not a multi-type Galton-Watson tree as described previously. We have, $\P^{\mathcal{E}}(N_e^{T_e^p}<p)>0$ and $\P^{\mathcal{E}}(\sum_{|x|=1}N_x^{T_e^p}\leq p)=1$. However, under $\P(\cdot|\mathcal{G}^{T_e^p}_{2})$, $(\{x\in\mathcal{R}_{T_e^p}, |x|\geq 2\},(N_x^{T_e^p};\; x\in\mathcal{R}_{T_e^p},|x|\geq 2))$ is distributed as $(\{x\in\mathcal{R}_{\tau^p},|x|\geq 2\},(N_x^{\tau^p};\; x\in\mathcal{R}_{\tau^p}),|x|\geq 2)$ under $\P(\cdot|\mathcal{G}^{\tau^p}_{2})$. Hence, with a few minor modifications, replacing $\mathcal{L}^{\lfloor tn\rfloor}$ in Theorem \ref{Th2} (resp. $T^{\lfloor\alpha n\rfloor}$ in Theorem \ref{Th1}) with $\mathcal{L}_e^{\lfloor tn\rfloor}$ (resp. $T_e^{\lfloor\alpha n\rfloor}$) does not affect the proofs.
\end{remark}

\noindent The multi-type Galton-Watson tree $(\mathcal{R}_{\tau^1},(N_x^{\tau^1};\; x\in\mathcal{R}_{\tau^1}))$ plays a key role in our study because: 
\begin{Fact}\label{GWF_Hyp}
   If $(\mathcal{R}_{1,i},(\beta_i(x);\; x\in\mathcal{R}_{1,i}))$, $i\in\N^*$, are independent copies of $(\mathcal{R}_{\tau^1},(N_x^{\tau^1}; x\in\mathcal{R}_{\tau^1}))$ under $\P$, then $\mathcal{F}^1:=((\mathcal{R}_{1,i},(\beta_i(x);\; x\in\mathcal{R}_{1,i})))_{i\in\N^*}$ is a multi-type Galton-Watson forest satisfying hypothesis \ref{H1}, \ref{H2} and \ref{H3}. Indeed, by definition, $\P$-almost surely, $N_e^{\tau^1}=1$ so \ref{H1} holds. For hypothesis \ref{H2} and \ref{H3}, we have, by definition, together with section 6.2 of \cite{AidRap} for $\kappa>2$ and Proposition 2 in \cite{deRaph1} for $\kappa\in(1,2]$, $\E[\sum_{x\in\mathcal{R}_{1,1}}\un_{\{G^1_{1}(x)=1,\; \beta_1(x)=1\}}]=\E[\sum_{x\in\T}\un_{\{N_x^{\tau^1}=1,\; \min_{e<z<x}N_y^{\tau^1}\geq 2\}}]=1$ and, also note that we have $\E[\sum_{x\in\mathcal{R}_{1,1}}\beta(x)\un_{\{G^1_1(x)=1\}}]=\E[\sum_{x\in\T}N_x^{\tau^1}\un_{\{\min_{e<z<x}N_y^{\tau^1}\geq 2\}}]=1/C_{\infty}<\infty$. Also, when $\kappa>2$, the variance of $\sum_{x\in\T}N_x^{\tau^1}\un_{\{\min_{e<z<x}N_y^{\tau^1}\geq 2\}}$ exists and is equal to $2(c_0)^2/C_{\infty}$. Finally, when $\kappa\in(1,2]$,  the limit $c_{\kappa}=\lim_{m\to\infty}m^{\kappa}\P(\sum_{x\in\T}\un_{\{N_x^{\tau^1}=1,\; \min_{e<z<x}N_y^{\tau^1}\geq 2\}}>m)$ exists, see \eqref{TailType1}.
\end{Fact}

\vspace{0.3cm}

\noindent Now, introduce $H_x:=\sum_{e\leq u\leq x}e^{V(u)-V(x)}$ and $T_x:=\inf\{k\geq 1,\; X_k=x\}$.
\begin{lemm}\label{LemmaEdge} It is very well known that $\P^{\mathcal{E}}(T_x<\tau^1)=e^{-V(x)}/H_x$ for any $x\in\T$, $x\not=e$, and we can also deduce from Fact \ref{FactGWMulti} the following
    \begin{align*}
        \E^{\mathcal{E}}\big[N_x^{\tau^1}\big]=e^{-V(x)},
    \end{align*}
    and for any $y\not=e$
    \begin{align*}
        \E^{\mathcal{E}}\big[N_x^{\tau^1}N_y^{\tau^1}\big]=e^{-V(y)}(2H_x-1)\;\;\textrm{ if } x\leq y,\;\; \E^{\mathcal{E}}\big[N_x^{\tau^1}N_y^{\tau^1}\big]=2H_{x\land y}e^{V(x\land y)}e^{-V(x)}e^{-V(y)}\;\;\textrm{ otherwise}.
    \end{align*}
\end{lemm}
\noindent See for example Lemma 3.6 in \cite{AndDiel3} for an explicit proof.

\vspace{0.2cm}

\noindent When studying the random walk $\X$, it is quite common to select only the vertices in the Galton-Watson tree $\T$ that we believe to be relevant. 

\vspace{0.3cm}

\noindent\textbf{Selection of vertices in the random environment $\mathcal{E}$}

\vspace{0.1cm}

\noindent 

\noindent Similarly to the reflecting barrier introduced by Y. Hu and Z. Shi in \cite{HuShi15}, let $\ell\in\N$, $\lambda\geq 0$ and define the following $(\ell,\lambda)$-regular line 
\begin{align}
    \mathcal{O}_{\ell,\lambda}:=\big\{x\in\T;\, 1\leq|x|\leq\ell:\; \max_{1\leq i\leq |x|}H_{x_i}\leq\lambda\big\}.
\end{align}
Also introduce for any $h>0$
\begin{align}
    \mathcal{V}^{h}_{\ell,\lambda}:=\{x\in\mathcal{O}_{\ell,\lambda};\; \underline{V}(x)\geq -h\},
\end{align}
with $\underline{V}(x):=\min_{1\leq i\leq |x|}V(x_i)$. We claim that the relevant vertices of the range are those in $\mathcal{V}^{h}_{\ell,\lambda}$ for some $\ell$ and $\lambda$:
\begin{lemm}\label{RegLine}
    There exists $\eta_0>0$ with $|\psi(\eta_0)|>0$ such that for all $r,h>0$, $\kappa>1$ and any $r_1>r/\kappa$
    \begin{align*}
        \liminf_{n\to\infty}\P\big(\mathcal{R}_{\tau^{n^{r}},\ell_n}\subset\mathcal{V}^h_{\ell_n,n^{r_1}}\big)\geq 1-\frac{e^{-\eta_0h}}{1-e^{-|\psi(\eta_0)|}},
    \end{align*}
    where $\mathcal{R}_{s,\ell}:=\{x\in\mathcal{R}_{s};\; 1\leq|x|\leq\ell\}$, $(\ell_n)$ is a non-decreasing sequence of positive integers such that $\lim_{n\to\infty}\ell_n=\infty$ but $\ell_n=o(n^{r_1\kappa-r})$ as $n\to\infty$.
\end{lemm}
\begin{proof}
First, note that since for any vertex $x\in\T$ visited by the random walk $\X$ before the instant $\tau^{n^{r}}$, its ancestor $x_1,\ldots,x_{|x|-1}$ are also visited before $\tau^{n^{r}}$, it is enough to show that
\begin{align*}
    \liminf_{n\to\infty}\P\big(\mathcal{R}_{\tau^{n^{r}},\ell_n}\subset\{x\in\T;\; 1\leq|x|\leq\ell_n,\; \underline{V}(x)\geq -h\textrm{ and }H_x\leq n^{r_1}\}\big)\geq 1-\frac{e^{-\eta_0h}}{1-e^{-|\psi(\eta_0)|}}.
\end{align*}
One can see that for any $n\in\N^*$
\begin{align*}
    \P\big(\exists\; j\leq \tau^{n^r};\; 1\leq|X_j|\leq\ell_n\textrm{ and }\underline{V}(X_j)<-h\big)\leq\Pb\big(\min_{x\in\T} V(x)<-h\big)\leq \frac{e^{-\eta_0h}}{1-e^{-|\psi(\eta_0)|}},
\end{align*}
where $\eta_0>0$ is chosen such that $\psi(\eta_0)<0$ which is possible since $\psi'(1)<0$. \\
For every environment $\mathcal{E}$
\begin{align*}
    \P^{\mathcal{E}}\big(\exists\; j\leq \tau^{n^{r}}:\; 1\leq|X_j|\leq\ell_n\textrm{ and }H_x>n^{r_1}\big)=\sum_{1\leq|x|\leq\ell_n}\un_{\{H_{x}>n^{r_1}\}}\P^{\mathcal{E}}\big(T_x<\tau^{n^{r}}\big).
\end{align*}
Then, note that $\P^{\mathcal{E}}(T_x<\tau^{n^{r}})=\sum_{j=1}^{n^{r}}\P^{\mathcal{E}}(\tau^{j-1}<T_x<\tau^j)$ and thanks to the strong Markov property at time $\tau^{j-1}$, $\P^{\mathcal{E}}(\tau^{j-1}<T_x<\tau^j)$ is smaller than $\P^{\mathcal{E}}(T_x<\tau^1)$ which is equal to $e^{-V(x)}/H_x$ by Lemma \ref{LemmaEdge} thus giving
\begin{align*}
    \P^{\mathcal{E}}\big(\exists\; j\leq \tau^{n^{r}}:\; 1\leq|X_j|\leq\ell_n\textrm{ and }H_x>n^{r_1}\big)&\leq\sum_{|x|\leq\ell_n}\un_{\{H_{x}>n^{r_1}\}}\P^{\mathcal{E}}\big(T_x<\tau^{n^{r}}\big) \\ & \leq n^r\sum_{|x|\leq\ell_n}\frac{e^{-V(x)}}{H_x}\un_{\{H_x>n^{r_1}\}} \\ & \leq\frac{1}{n^{r_1-r}}\sum_{|x|\leq\ell_n}e^{-V(x)}\un_{\{H_x>n^{r_1}\}}.
\end{align*}
Now, thanks to the many-to-one Lemma \eqref{ManyTo1}
\begin{align*}
    \P\big(\exists\; j\leq \tau^{n^{r}}:\; 1\leq|X_j|\leq\ell_n\textrm{ and }H_x>n^{r_1}\big)\leq\frac{1}{n^{r_1-r}}\sum_{k=0}^{\ell_n}\Pb(H^S_{k}>n^{r_1}),
\end{align*}
where $H^S_k:=\sum_{i=0}^{k}e^{-S_i}$. By (\cite{AndDiel3}, lemma 2.2), there exists a constant $C_{\ref{RegLine}}>0$ such that $\Pb(H^S_{j}>n^{r_1})\leq C_{\ref{RegLine}}/n^{r_1(\kappa-1)}$ so we finally obtain
\begin{align*}
    \P\big(\exists\; j\leq \tau^{n^{r}}:\; 1\leq|X_j|\leq\ell_n\textrm{ and }H_x>n^{r_1}\big)\leq\frac{C_{\ref{RegLine}}\ell_n}{n^{\kappa r_1-r}},
\end{align*}
which goes to $0$ by definition of $\ell_n$ and this completes the proof.
\end{proof}

\vspace{0.3cm}

\noindent\textbf{The subset $\mathcal{B}^p_{\ell}$ of $\T$}

\vspace{0.1cm}

\noindent For any $\ell\in\N$ and $p\in\N^*$, define 
\begin{align}
    \mathcal{B}^{p}_{\ell}:=\big\{x\in\T,\; |x|>\ell;\; N_x^{\tau^p}=1\textrm{ and }\min_{\ell<i<|x|}N_{x_i}^{\tau^p}\geq 2\big\},
\end{align}
with the convention $\mathcal{B}^0_{\ell}=\varnothing$. This set of vertices, first introduced by E. Aïdékon and L. de Raphélis in \cite{AidRap}, plays a key role in our study, so before going any further, let us give some facts and properties about it. The following fact, due to E. Aïdékon and L. de Raphélis (Lemma 7.2 in \cite{AidRap}) and L. de Raphélis (Lemma 22 in \cite{deRaph1}), says that any vertex in $\T$ in a generation large enough has an ancestor visited as most once by the random walk $\X$:
\begin{Fact}\label{PropSetRoots}
    Let $1<s<s_1$ and $|z|>(\log n)^{s_1}$. If $\mathfrak{t}^{p}_z:=\inf\{(\log n)^s<\ell<(\log n)^{s_1};\; N_{z_{\ell}}^{\tau^p}\in\{0,1\}\}$ denotes the first generation larger than $(\log n)^s$ at which the vertex $z$ has an ancestor $u$ such that the oriented edge $(u^*,u)$ is visited at most once by the random walk $\X$ before $\tau^p$, then, for any $r_2>0$, $\Pb$-almost surely
    \begin{align*}
        \P^{\mathcal{E}}\big(\forall\; p\leq n^{r_2},\; \forall\; |z|>(\log n)^{s_1},\; \mathfrak{t}^{p}_z<\infty\big)\underset{n\to\infty}{\longrightarrow} 1.
    \end{align*}
\end{Fact}
\noindent We first deduce from Fact \ref{PropSetRoots} that, with large probability, $(\mathcal{B}^{p}_{(\log n)^2})_{p\leq n^{r}}$ is a non-decreasing sequence of sets for any $r>0$ in the following: 
\begin{lemm}\label{NDSet} For any $r>0$, $\Pb$-almost surely
    \begin{align*}
    \lim_{n\to\infty}\P^{\mathcal{E}}\big((\mathcal{B}^{p}_{(\log n)^2})_{p\leq n^{r}}\textrm{ is non-decreasing}\big)=1,
\end{align*}
where $(\mathcal{B}^{p}_{(\log n)^2})_{p\leq n^{r}}$ is non-decreasing means here that for any $p,p'\in\{1,\ldots,n^r\}$, $p\leq p'$ implies that $\mathcal{B}^{p}_{(\log n)^2}\subset\mathcal{B}^{p'}_{(\log n)^2}$.
\end{lemm}
\noindent Note that, by definition, if $\mathcal{B}^p_{\ell}$ contains at least two vertices $x$ and $y$, then $x\leq y$ or $y\leq x$ if and only if $x=y$ so the sequence $(\mathcal{B}^{p}_{(\log n)^2})_{p\leq n^{r}}$ grows horizontally.
\begin{proof}
    Let $p\leq n^r$. If $\mathcal{B}^p_{(\log n)^2}$ is empty, then it is included in $\mathcal{B}^{p+1}_{(\log n)^2}$. Otherwise, let $z\in\mathcal{B}^{p}_{(\log n)^2}$ and introduce, for any $x\in\T$, $E^p_x:=\sum_{j=1}^p\un_{\{N_x^{\tau^j}-N_x^{\tau^{j-1}}\geq 1\}}$, the number of excursions above $e^*$ during which the vertex $x$ is visited up to $\tau^p$. On the event  $\{\forall\; |x|>(\log n)^2\textrm{ and }E^{n^r}_x\in\{0,1\}\}$, there exists $j\in\{1,\ldots,p\}$ such that $N_{z}^{\tau^j}-N_{z}^{\tau^{j-1}}=1$ and for all $j'\in\{1,\ldots,n^r\}\setminus\{j\}$, $N_{z}^{\tau^{j'}}-N_{z}^{\tau^{j'-1}}=0$. Then, $N_z^{\tau^{p+1}}=\sum_{j'=1;j'\not=j}^{p+1}(N_z^{\tau^{j'}}-N_z^{\tau^{j'-1}})+N_z^{\tau^j}-N_z^{\tau^{j-1}}=N_z^{\tau^j}-N_z^{\tau^{j-1}}=1$. Moreover, it is easy to see that for all $(\log n)^2\leq i<|z|$, $N_{z_i}^{\tau^{p+1}}\geq N_{z_i}^{\tau^{p}}\geq 2$. In other words, $z$ belongs to $\mathcal{B}^{p+1}_{(\log n)^2}$. It follows that
    \begin{align*}
        \P^{\mathcal{E}}\big((\mathcal{B}^{p}_{(\log n)^2})_{p\leq n^{r}}\textrm{ is non-decreasing}\big)\geq\P^{\mathcal{E}}\big(\forall\; |x|>(\log n)^2\textrm{ and }E^{n^r}_x\in\{0,1\}\big),
    \end{align*}
    and the latter probability goes to $1$ when $n\to\infty$ thanks to Fact \ref{PropSetRoots} with $r_2=r$ and $s_1=2$.
\end{proof}
\noindent From Fact \ref{PropSetRoots}, we also deduce the following important properties:
\begin{lemm}\label{PropSetRoots1} \phantom{à la ligne}
    \begin{enumerate}[label=(\roman*)]
        \item\label{Fact2.1} $\mathcal{B}^p_{\ell}$ is non-empty for $p$ and $\ell$ large enough: 
        for any $0<r_1<r_2$ and any $s>0$, $\Pb$-almost surely
        \begin{align*}
            \lim_{n\to\infty}\P^{\mathcal{E}}\big(\forall\; n^{r_1}\leq p\leq n^{r_2},\; \mathcal{B}^{p}_{(\log n)^{s}}\not=\varnothing\big)=1;
        \end{align*}
        \item\label{Fact2.2}  $\Pb$-almost surely
        \begin{align*}
             \lim_{n\to\infty}\P^{\mathcal{E}}\Big(\forall\; n^{r_1}\leq p\leq n^{r_2},\;\sup_{x\in\mathcal{B}^{p}_{(\log n)^2}}|x|\leq(\log n)^3\Big)=1.
        \end{align*}
    \end{enumerate}
\end{lemm}
\begin{proof}
    \ref{Fact2.1} and \ref{Fact2.2} are consequences of Fact \ref{PropSetRoots}. Indeed, for any $p\geq n^{r_1}$, we have $\tau^p\geq\tau^{n^{r_1}}\geq n^{r_1}$ so 
    \begin{align}\label{Fact2.1.2}
    &\P^{\mathcal{E}}\big(\forall\; p\geq n^{r_1},\;\exists\; |u|>(\log n)^{s_1}:\; T_u<\tau^p\big)\geq \P^{\mathcal{E}}\big(\max_{j\leq n^{r_1}}|X_j|>(\log n)^{s_1}\big)\underset{n\to\infty}{\longrightarrow} 1.
\end{align}
In other words, any vertex of the tree $\T$, in a generation larger than $(\log n)^{s_1}$, is visited during the first $p\geq n^{r_1}$ excursions of $\X$ above $e^*$. Combining \eqref{Fact2.1.2} with Fact \ref{PropSetRoots}, we obtain \ref{Fact2.1} and finally, \eqref{Fact2.1.2} and \ref{PropSetRoots} with $s_1=3$ and $s=2$ yield \ref{Fact2.2} and the proof is completed.
\end{proof}
\noindent We know that $(\mathcal{B}^p_{(\log n)^2})_{p\leq n^r}$ is a non-decreasing sequence of non-empty sets (when $p$ is large enough), we now focus our attention to the cardinal $B^p_{\ell}:=\sum_{|x|>\ell}\un_{\{N_x^{\tau^p}=1,\; \min_{\ell<i<|x|}N_{x_i}^{\tau^p}\geq 2\}}$ of the set $\mathcal{B}^p_{\ell}$ and by convention, $B^0_{\ell}=0$. Let us first state a fact that will be useful in the following:
\begin{Fact}[Lemma 9 and Lemma 14 in \cite{deRaph1}]\label{CardSetRoot} 
For any $m\in\N^*$
\begin{enumerate}[label=(\roman*)]
    \item\label{CardSetRoot1}
        \begin{align*}
            \E[B_0^{m}]=m;
        \end{align*}
    \item\label{CardSetRoot2} for any $\xi\in(1,\kappa)$
        \begin{align*}
            \E[(B_0^{m})^{\xi}]\leq C_{\ref{CardSetRoot}}m^{\xi},
        \end{align*}
        for some constant $C_{\ref{CardSetRoot}}$ only depending on $\xi$.
\end{enumerate}

\end{Fact}
\noindent The next lemma gives an uniform estimation of $B^p_{\ell}$ when $p$ and $\ell$ are large enough:
\begin{lemm}\label{RootNumber} Recall that $W_{\infty}$ is the limit of the additive martingale $(\sum_{|x|=\ell}e^{-V(x)})_{\ell}$ and $W_{\infty}$ is positive $\Pb^*$-almost surely. For any $M,r,\varepsilon_1>0$
\begin{align*}
\lim_{n\to\infty} \P^*\Big(\sup_{\alpha\in[0,M]}\Big|\frac{1}{n^r}B^{\lfloor\alpha n^r\rfloor}_{(\log n)^2}-\alpha W_{\infty}\Big|>\varepsilon_1\Big)=0.
\end{align*}
\end{lemm}
\noindent Note that if we would consider $T^p_e$ (the $p$-th return time to the root $e$ of the tree $\T$) instead of $\tau^p$ in the definition of $B^p_{\ell}$, then we would obtain a similar result but $W_{\infty}$ would be replaced by $p^{\mathcal{E}}(e,e^*)W_{\infty}$.

\begin{proof}
For any vertex $u\in\T$, introduce
\begin{align}\label{DefSetRootTransla}
    \Tilde{\mathcal{B}}^{p}_u:=\big\{x\in\T,\; x>u;\; N_x^{\tau^p}=1\textrm{ and }\min_{|u|<i<|x|}N_{x_i}^{\tau^p}\geq 2\big\},
\end{align}
and we denote by $\Tilde{B}^p_u$ the cardinal of $\Tilde{\mathcal{B}}^{p}_u$: $\Tilde{B}^p_u=\sum_{x>u}\un_{\{N_x^{\tau^p}=1,\; \min_{|u|<i<|x|}N_{x_i}^{\tau^p}\geq 2\}}$. Let $x\in\T$ and $u<x$. Define the translated potential by $V_u(x):=V(x)-V(u)$,\; $\underline{V}_u(x):=\min_{u<z\leq x}V_u(z)$ and the translated exponential potential downfall by $H_{u,x}:=\sum_{u\leq z\leq x}e^{V_u(z)-V_u(x)}$. Also introduce the translated $(\ell,\lambda)$-regular line $\mathcal{O}^u_{\ell,\lambda}$ by $\mathcal{O}^u_{\ell,\lambda}:=\{x\in\T,\; x>u,\; |x|\leq\ell;\; \max_{u<z\leq x}H_z\leq\lambda\}$ and the translated version of $\mathcal{V}^{h}_{\ell,\lambda}$ by $\mathcal{V}^{h,u}_{\ell,\lambda}:=\{x\in\mathcal{O}^{u}_{\ell,\lambda};\; \underline{V}_u(x)\geq-h\}$. The main reason why we introduce these sets of vertices is that it somehow allows to have a good control of $\E[(B^m_0)^2]$, even when $\kappa$ is close to $1$. \\
We split the proof into three main steps. In the first step, we show that for any $\alpha, h, r>0$
\begin{align}\label{EspPart1}
    \lim_{h\to+\infty}\;\limsup_{n\to\infty}\frac{1}{n^r}\E\Big[\Big|\sum_{|u|=(\log n)^2}\un_{\{u\in\mathcal{V}^h_{(\log n)^2,\lambda_n}\}}\un_{\{N_u^{\tau^{p_n}}\leq hp_n\}}(\Tilde{B}^{p_n}_u-N_u^{\tau^{p_n}})\Big|\Big]=0,
\end{align}
with $p_n:=\lfloor\alpha n^r\rfloor$ and $\lambda_n=n^{r(1+1/\kappa)/2}$. \\
Indeed, for any $h>0$, $\E[|\sum_{|u|=(\log n)^2}\un_{\{u\in\mathcal{V}^h_{(\log n)^2,\lambda_n}\}}\un_{\{N_u^{\tau^{p_n}}\leq hp_n\}}(\Tilde{B}^{p_n}_u-N_u^{\tau^{p_n}})|]$ is smaller than
\begin{align}\label{DecoupeEsp}
    &\E\Big[\sum_{|u|=(\log n)^2}\big(\Tilde{B}^{p_n}_{u}-\Tilde{B}^{p_n}_{u,\lambda_n}\big)\un_{\{u\in\mathcal{V}^h_{(\log n)^2,\lambda_n}\}}\un_{\{N_u^{\tau^{p_n}}\leq hp_n\}}\Big]\nonumber \\ & +\E\Big[\Big|\sum_{|u|=(\log n)^2}\big(\Tilde{B}^{p_n}_{u,\lambda_n}-N_u^{\tau^{p_n}}\big)\un_{\{u\in\mathcal{V}^h_{(\log n)^2,\lambda_n}\}}\un_{\{N_u^{\tau^{p_n}}\leq hp_n\}}\Big|^2\Big]^{1/2},
\end{align}
where $\Tilde{B}^{p_n}_{u,\lambda_n}:=\sum_{x\in\mathcal{V}^{h,u}_{2(\log n)^2,\lambda_n}}\un_{\{N_x^{\tau^{p_n}}=1,\; \min_{|u|<i<|x|}N_{x_i}^{\tau^{p_n}}\geq 2\}}$. \\
We first deal with $\E[\sum_{|u|=(\log n)^2}\big(\Tilde{B}^{p_n}_{u}-\Tilde{B}^{p_n}_{u,\lambda_n}\big)\un_{\{u\in\mathcal{V}^h_{(\log n)^2,\lambda_n}\}}\un_{\{N_u^{\tau^{p_n}}\leq hp_n\}}]$. By Fact \ref{CardSetRoot} \ref{CardSetRoot1}, for any $m\in\N^*$, we have $m=\E[\sum_{|x|\geq 1}\un_{\{N_x^{\tau^m}=1,\; \min_{1\leq i<|x|}N_{x_i}^{\tau^m}\geq 2\}}]$ and thanks to the mixed branching property \eqref{BP}
\begin{align*}
    \E\Big[\sum_{|u|=(\log n)^2}\big(\Tilde{B}^{p_n}_{u}-\Tilde{B}^{p_n}_{u,\lambda_n}\big)\un_{\{u\in\mathcal{V}^h_{(\log n)^2,\lambda_n}\}}\un_{\{N_u^{\tau^{p_n}}\leq hp_n\}}\Big]=\E\Big[\sum_{|u|=(\log n)^2}\Bar{\Psi}^h_n(N_u^{\tau^{p_n}})\un_{\{u\in\mathcal{V}^h_{(\log n)^2,\lambda_n}\}}\Big],
\end{align*}
where $\Bar{\Psi}^h_n(m):=(m-\Psi^h_n(m))\un_{\{m\leq hp_n\}}$ and $\Psi^h_n(m):=\E[\sum_{x\in\mathcal{V}^h_{(\log n)^2,\lambda_n},|x|\geq 1}\un_{\{N_x^{\tau^{m}}=1,\; \min_{1\leq i<|x|}N_{x_i}^{\tau^{m}}\geq 2\}}]$. Note that thanks to the Hölder inequality, for any $\zeta\in(0,1-1/\kappa)$
\begin{align*}
    \Bar{\Psi}^h_n(m)&\leq\E\Big[\sum_{|x|\geq 1}\un_{\{N_x^{\tau^{m}}=1,\; \min_{1\leq i<|x|}N_{x_i}^{\tau^{m}}\geq 2\}}\un_{\{\mathcal{R}_{\tau^m,(\log n)^2}\not\subset\mathcal{V}^{h}_{(\log n)^2,\lambda_n}\}}\Big]\un_{\{m\leq hp_n\}} \\ & \leq \E\Big[\Big(\sum_{|x|\geq 1}\un_{\{N_x^{\tau^{m}}=1,\; \min_{1\leq i<|x|}N_{x_i}^{\tau^{m}}\geq 2\}}\Big)^{\zeta'}\Big]^{1/\zeta'}\P\big(\mathcal{R}_{\tau^m, (\log n)^2}\not\subset\mathcal{V}^{h}_{(\log n)^2,\lambda_n}\big)^{\zeta}\un_{\{m\leq hp_n\}},
\end{align*}
with $\zeta'=1/(1-\zeta)$. By Fact \ref{CardSetRoot} \ref{CardSetRoot2}, $\E[(\sum_{|x|\geq 1}\un_{\{N_x^{\tau^{m}}=1,\; \min_{1\leq i<|x|}N_{x_i}^{\tau^{m
}}\geq 2\}})^{\zeta'}]\leq(C_{\ref{CardSetRoot}}m)^{\zeta'}$ thus giving
\begin{align*}
    0\leq\Bar{\Psi}^h_n(m)\leq C_{\ref{CardSetRoot}}m\P\big(\mathcal{R}_{\tau^m, (\log n)^2}\not\subset\mathcal{V}^{h}_{(\log n)^2,\lambda_n}\big)^{\zeta}\un_{\{m\leq hp_n\}}\leq C_{\ref{CardSetRoot}}m\P\big(\mathcal{R}_{\tau^{hp_n}, (\log n)^2}\not\subset\mathcal{V}^{h}_{(\log n)^2,\lambda_n}\big)^{\zeta},
\end{align*}
where we have used, for the second inequality, that $m\leq hp_n$ implies $\mathcal{R}_{\tau^{m}}\subset\mathcal{R}_{\tau^{hp_n}}$. This yields
\begin{align*}
    \E\Big[\sum_{|u|=(\log n)^2}\big(\Tilde{B}^{p_n}_{u}-\Tilde{B}^{p_n}_{u,\lambda_n}\big)\un_{\{u\in\mathcal{V}^h_{(\log n)^2,\lambda_n}\}}\un_{\{N_u^{\tau^{p_n}}\leq hp_n\}}\Big]\leq C_{\ref{CardSetRoot}}&\P\big(\mathcal{R}_{\tau{hp_n}, (\log n)^2}\not\subset\mathcal{V}^{h}_{(\log n)^2,\lambda_n}\big)^{\zeta} \\ & \times\E\Big[\sum_{|u|=(\log n)^2}N_u^{\tau^{p_n}}\un_{\{u\in\mathcal{V}^h_{(\log n)^2,\lambda_n}\}}\Big].
\end{align*}
Note that $\E[\sum_{|u|=(\log n)^2}N_u^{\tau{p_n}}\un_{\{u\in\mathcal{V}^h_{(\log n)^2,\lambda_n}\}}]\leq p_n\Eb[\sum_{|u|=(\log n)^2}e^{-V(u)}]=p_n$, so Lemma \ref{RegLine} leads to
\begin{align}\label{EspPart11}
    &\limsup_{n\to\infty}\frac{1}{n^r}\E\Big[\sum_{|u|=(\log n)^2}\big(\Tilde{B}^{p_n}_{u}-\Tilde{B}^{p_n}_{u,\lambda_n}\big)\un_{\{u\in\mathcal{V}^h_{(\log n)^2,\lambda_n}\}}\un_{\{N_u^{\tau^{p_n}}\leq hp_n\}}\Big]\leq C_{\ref{RootNumber},1}e^{-\eta_0h\zeta},
\end{align}
for some constant $C_{\ref{RootNumber},1}>0$. \\
We now deal with $\E[|\sum_{|u|=(\log n)^2}(\Tilde{B}^{p_n}_{u,\lambda_n}-N_u^{\tau{p_n}})\un_{\{u\in\mathcal{V}^h_{(\log n)^2,\lambda_n}\}}\un_{\{N_u^{\tau^{p_n}}\leq hp_n\}}|^2]$. Note that this mean is equal to
\begin{align*}
    &\E\Big[\sum_{|u|=|v|=(\log n)^2}\Tilde{B}^{p_n}_{u,\lambda_n}\Tilde{B}^{p_n}_{v,\lambda_n}\un_{\{u,v\in\mathcal{V}^h_{(\log n)^2,\lambda_n}\}}\un_{\{N_u^{\tau^{p_n}}\leq hp_n,\; N_v^{\tau^{p_n}}\leq hp_n\}}\Big] \\ & +\E\Big[\sum_{|u|=|v|=(\log n)^2}\big(N_u^{\tau{p_n}}-2\Tilde{B}^{p_n}_{u,\lambda_n}\big)N_v^{\tau{p_n}}\un_{\{u,v\in\mathcal{V}^h_{(\log n)^2,\lambda_n}\}}\un_{\{N_u^{\tau^{p_n}}\leq hp_n,\; N_v^{\tau^{p_n}}\leq hp_n\}}\Big].
\end{align*}
One can notice that $\E[\sum_{|u|=|v|=(\log n)^2}\Tilde{B}^{p_n}_{u,\lambda_n}\Tilde{B}^{p_n}_{v,\lambda_n}\un_{\{u,v\in\mathcal{V}^h_{(\log n)^2,\lambda_n}\}}\un_{\{N_u^{\tau^{p_n}}\leq hp_n,\; N_v^{\tau^{p_n}}\leq hp_n\}}]$ is smaller than
\begin{align*}
    \E\Big[\sum_{|u|=(\log n)^2}(\Tilde{B}^{p_n}_{u,\lambda_n})^2\un_{\{u\in\mathcal{V}^h_{(\log n)^2,\lambda_n}\}}\Big]+\E\Big[\underset{|u|=|v|=(\log n)^2}{\sum_{u\not=v}}\Tilde{B}^{p_n}_{u,\lambda_n}\Tilde{B}^{p_n}_{v,\lambda_n}\un_{\{u,v\in\mathcal{V}^h_{(\log n)^2,\lambda_n}\}}\Big],
\end{align*}
and thanks to the mixed branching property \eqref{BP}, we have 
\begin{align*}
    &\E\Big[\underset{|u|=|v|=(\log n)^2}{\sum_{u\not=v}}\Tilde{B}^{p_n}_{u,\lambda_n}\Tilde{B}^{p_n}_{v,\lambda_n}\un_{\{u,v\in\mathcal{V}^h_{(\log n)^2,\lambda_n}\}}\Big]\leq \E\Big[\underset{|u|=|v|=(\log n)^2}{\sum_{u\not=v}}N_u^{\tau^{p_n}}N_v^{\tau^{p_n}}\un_{\{u,v\in\mathcal{V}^h_{(\log n)^2,\lambda_n}\}}\Big].
\end{align*}
This ensures that $\E[|\sum_{|u|=(\log n)^2}(\Tilde{B}^{p_n}_{u,\lambda_n}-N_u^{\tau{p_n}})\un_{\{u\in\mathcal{V}^h_{(\log n)^2,\lambda_n}\}}\un_{\{N_u^{\tau^{p_n}}\leq hp_n\}}|^2]$ is smaller than
\begin{align}\label{MoyenneCarre}
    &\E\Big[\sum_{|u|=(\log n)^2}(\Tilde{B}^{p_n}_{u,\lambda_n})^2\un_{\{u\in\mathcal{V}^h_{(\log n)^2,\lambda_n}\}}\Big]\nonumber \\ & +2\E\Big[\sum_{|u|=|v|=(\log n)^2}\big(N_u^{\tau{p_n}}-\Tilde{B}^{p_n}_{u,\lambda_n}\big)N_v^{\tau{p_n}}\un_{\{u,v\in\mathcal{V}^h_{(\log n)^2,\lambda_n}\}}\un_{\{N_u^{\tau^{p_n}}\leq hp_n\}}\Big].
\end{align}
We first prove that $\E[\sum_{|u|=(\log n)^2}(\Tilde{B}^{p_n}_{u,\lambda_n})^2\un_{\{u\in\mathcal{V}^h_{(\log n)^2,\lambda_n}\}}]=o(n^{2r})$. For that, we decompose $(\Tilde{B}^{p_n}_{u,\lambda_n})^2$ as follows
\begin{align*}
    (\Tilde{B}^{p_n}_{u,\lambda_n})^2=\Big(\sum_{x\in\mathcal{V}^{h,u}_{2(\log n)^2,\lambda_n}}\un_{\{N_x^{\tau^{p_n}}=1,\; \min_{|u|\leq i<|x|}N_{x_i}^{\tau^{p_n}}\geq 2\}}\Big)^2=\Tilde{B}^{p_n}_{u,\lambda_n}+\sum_{x\not=y\in\mathcal{V}^{h,u}_{2(\log n)^2,\lambda_n}}\un_{\{x,y\in\Tilde{\mathcal{B}}^{p_n}_u\}},
\end{align*}
where we recall the definition of $\Tilde{\mathcal{B}}^{p_n}_u$ in \eqref{DefSetRootTransla}. Thanks to the mixed branching property \eqref{BP} and Fact \ref{CardSetRoot} \ref{CardSetRoot1}, we have $\E[\sum_{|u|=(\log n)^2}\Tilde{B}^{p_n}_{u,\lambda_n}]\leq\E[\sum_{|u|=(\log n)^2}N_u^{\tau^{p_n}}]=p_n\Eb[\sum_{|u|=(\log n)^2}e^{-V(u)}]=p_n=o(n^{2r})$. For the second part of the decomposition, note that, for any $x\not=y$ in $\mathcal{V}^{h,u}_{2(\log n)^2,\lambda_n}$, $u$ is a common ancestor of $x$ and $y$ so $u\leq x\land y$ and if $u<x\land y$, then $x\land y$ also belongs to $\mathcal{V}^{h,u}_{2(\log n)^2,\lambda_n}$. Moreover, by definition, if $x\in \Tilde{\mathcal{B}}^{p_n}_u$, then for any ancestor $v$ of $x$ such that $u\leq v<x$, we have $x\in\Tilde{\mathcal{B}}^{p_n}_v$. It follows
\begin{align*}
    &\sum_{|u|=(\log n)^2}\un_{\{u\in\mathcal{V}^{h}_{(\log n)^2,\lambda_n}\}}\sum_{x\not=y\in\mathcal{V}^{h,u}_{2(\log n)^2,\lambda_n}}\un_{\{x,y\in\Tilde{\mathcal{B}}^{p_n}_u,\; x\land y>u\}} \\ & \leq\sum_{|u|=(\log n)^2}\un_{\{u\in\mathcal{V}^{h}_{(\log n)^2,\lambda_n}\}}\sum_{k=(\log n)^2+1}^{2(\log n)^2-2}\;\underset{z>u}{\sum_{|z|=k}}\un_{\{z\in\mathcal{V}^{h,u}_{2(\log n)^2,\lambda_n}\}}\underset{v^*=w^*=z}{\sum_{v\not=w}}\Tilde{B}^{p_n}_v\Tilde{B}^{p_n}_w \\ & + \sum_{|u|=(\log n)^2}\un_{\{u\in\mathcal{V}^{h}_{(\log n)^2,\lambda_n}\}}\underset{z>u}{\sum_{|z|=2(\log n)^2-1}}\un_{\{z\in\mathcal{V}^{h,u}_{2(\log n)^2,\lambda_n}\}}\underset{x^*=y^*=z}{\sum_{x\not=y}}\un_{\{x,y\in\Tilde{\mathcal{B}}^{p_n}_z\}} \\ & + \sum_{|u|=(\log n)^2}\un_{\{u\in\mathcal{V}^{h}_{(\log n)^2,\lambda_n}\}}\sum_{x\not=y\in\T;\; x\land y=u}\un_{\{x,y\in\Tilde{\mathcal{B}}^{p_n}_u\}}.
\end{align*}
Note that we had to distinguish the case $|x\land y|<2(\log n)^2-1$ from the case $x^*=y^*$ when $x\land y>u$. The mixed branching property \eqref{BP}, Fact \ref{CardSetRoot} \ref{CardSetRoot1} and the definition of $\Tilde{\mathcal{B}}^{p_n}_{\cdot}$ yield
\begin{align*}
    &\E\Big[\sum_{|u|=(\log n)^2}\un_{\{u\in\mathcal{V}^{h}_{(\log n)^2,\lambda_n}\}}\sum_{x\not=y\in\mathcal{V}^{h,u}_{(\log n)^2,\lambda_n}}\un_{\{x,y\in\Tilde{\mathcal{B}}^{p_n}_u\}}\Big]\leq \Sigma_n+\Tilde{\Sigma}_n,
\end{align*}    
where
\begin{align*}
    \Sigma_n:=\E\Big[\sum_{|u|=(\log n)^2}\un_{\{u\in\mathcal{V}^{h}_{(\log n)^2,\lambda_n}\}}\sum_{k=(\log n)^2+1}^{2(\log n)^2-2}\;\underset{z>u}{\sum_{|z|=k}}\un_{\{z\in\mathcal{V}^{h,u}_{2(\log n)^2,\lambda_n}\}}\underset{v^*=w^*=z}{\sum_{v\not=w}}N_v^{T^{p_n}}N_w^{T^{p_n}}\Big],
\end{align*}
and
\begin{align*}
    \Tilde{\Sigma}_n:=\E\Big[\sum_{|u|=(\log n)^2}\un_{\{u\in\mathcal{V}^{h}_{(\log n)^2,\lambda_n}\}}\underset{z>u}{\sum_{|z|=2(\log n)^2-1}}\un_{\{z\in\mathcal{V}^{h,u}_{2(\log n)^2,\lambda_n}\}}\varphi_1\big(N_z^{\tau^{p_n}}\big)\Big],
\end{align*}
with $\varphi_1(m):=\E[\sum_{x\not=y;\; |x|=|y|=1}\un_{\{N_x^{\tau^m}=N_y^{\tau^m}=1\}}]$. \\ 
We first deal with $\Sigma_n$. We decompose $N_v^{\tau^{p_n}}N_w^{\tau^{p_n}}$ whether or not $v$ and $w$ are visited during the same excursion above $e^*$: $N_v^{\tau^{p_n}}N_w^{\tau^{p_n}}=\sum_{j=1}^{p_n}(N_v^{\tau^{j}}-N_v^{\tau^{j-1}})(N_w^{\tau^{j}}-N_w^{\tau^{j-1}})+\sum_{i\not=j=1}^{p_n}(N_v^{\tau^{j}}-N_v^{\tau^{j-1}})(N_w^{\tau^{i}}-N_w^{\tau^{i-1}})$ and thanks to the strong Markov property $\E^{\mathcal{E}}[N_v^{\tau^{p_n}}N_w^{\tau^{p_n}}]=2p_n\E^{\mathcal{E}}[N_v^{\tau^1}N_w^{\tau^1}]+p_n(p_n-1)\E^{\mathcal{E}}[N_v^{\tau^1}]\E^{\mathcal{E}}[N_w^{\tau^1}]$, using Lemma \ref{LemmaEdge} yields, $\E^{\mathcal{E}}[N_v^{\tau^{p_n}}N_w^{\tau^{p_n}}]\leq 2p_nH_ze^{-V(z)}e^{-V_z(v)-V_z(w)}+p_n^2e^{-2V(z)}e^{-V_z(v)-V_z(w)}$. It follows that $\Sigma_n$ is smaller than $\Sigma_{n,1}+\Sigma_{n,2}$ where
\begin{align*}
    \Sigma_{n,1}:=2c_1p_n\Eb\Big[\sum_{|u|=(\log n)^2}\un_{\{u\in\mathcal{V}^{h}_{(\log n)^2,\lambda_n}\}}\sum_{k=(\log n)^2+1}^{2(\log n)^2-2}\;\underset{z>u}{\sum_{|z|=k}}H_ze^{-V(z)}\un_{\{z\in\mathcal{V}^{h,u}_{2(\log n)^2,\lambda_n}\}}\Big],
\end{align*}
and
\begin{align*}
    \Sigma_{n,2}:=c_1p_n^2\Eb\Big[\sum_{|u|=(\log n)^2}\un_{\{u\in\mathcal{V}^{h}_{(\log n)^2,\lambda_n}\}}\sum_{k=(\log n)^2+1}^{2(\log n)^2-2}\;\underset{z>u}{\sum_{|z|=k}}e^{-2V(z)}\un_{\{z\in\mathcal{V}^{h,u}_{2(\log n)^2,\lambda_n}\}}\Big],
\end{align*}
where $c_1:=\Eb[\sum_{x\not=y;\; |x|=|y|=1}e^{-V(x)-V(y)}]$. Observe that, for any $z>u$, if $u\in\mathcal{V}^h_{(\log n)^2,\lambda_n}$ and $z\in\mathcal{V}^{h,u}_{2(\log n)^2,\lambda_n}$, then $V_{u}(z)\geq-h$ and $H_{u}\lor H_{u,x}\leq\lambda_n$ so $H_z=(H_{u}-1)e^{-V_{u}(z)}+H_{u,z}\leq(1+e^{h})\lambda_n$, where we recall that $\lambda_n=n^{r(1+1/\kappa)/2}$ and $p_n=\lfloor\alpha n^r\rfloor$. Hence
\begin{align*}
    \Sigma_{n,1}\leq 2(1+e^h)c_1p_n\lambda_n\Eb\Big[\sum_{|u|=(\log n)^2}\sum_{k=(\log n)^2+1}^{2(\log n)^2-2}\;\underset{z>u}{\sum_{|z|=k}}e^{-V(z)}\Big]\leq 2(1+e^h)\alpha c_1 n^{r+r(1+1/\kappa)/2}(\log n)^2.
\end{align*}
For $\Sigma_{n,2}$, since $\psi'(1)<0$ and $\kappa>1$, there exists $\delta_2\in(1,2)$ such that $\psi(\delta_2)\in(-\infty,0)$. Moreover, for $z\in\mathcal{V}^{h,u}_{2(\log n)^2,\lambda_n}$ and $u\in\mathcal{V}^h_{(\log n)^2,\lambda_n}$, $2V(z)=2V_u(z)+2V(u)\geq \delta_2V_u(z)+\delta_2V(z)-(2-\delta_2)h$ so
\begin{align*}
    \Sigma_{2,n}\leq c_1p_n^2e^{(2-\delta_2)h}\Eb\Big[\sum_{|u|=(\log n)^2}e^{-\delta_2V(u)}\sum_{k=(\log n)^2+1}^{2(\log n)^2-2}\;\underset{z>u}{\sum_{|z|=k}}e^{-\delta_2V_u(z)}\Big]\leq\frac{c_1e^{(2-\delta_2)h}}{1-e^{-|\psi(\delta_2)|}}p_n^2e^{-(\log n)^2|\psi(r_2)|}.
\end{align*}
Since $r+r(1+1/\kappa)/2<2r$, we deduce from these two inequalities that $\Sigma_n=o(n^{2r})$ as $n\to\infty$. \\
We now deal with $\Tilde{\Sigma}_n$. One can notice that for all $m\in\N^*$
\begin{align*}
    \varphi_1(m)=\E\Big[\underset{|x|=|y|=1}{\sum_{x\not=y}}\un_{\{N_x^{\tau^m}=N_y^{\tau^m}=1\}}\Big]\leq\E\Big[\underset{|x|=|y|=1}{\sum_{x\not=y}}N_x^{\tau^m}N_y^{\tau^m}\Big]\leq 3c_1m^2,
\end{align*}
where we have used Lemma \ref{LemmaEdge} for the last inequality so
\begin{align*}
    \Tilde{\Sigma}_n\leq 3c_1\E\Big[\sum_{|u|=(\log n)^2}\un_{\{u\in\mathcal{V}^{h}_{(\log n)^2,\lambda_n}\}}\underset{z>u}{\sum_{|z|=2(\log n)^2-1}}\un_{\{z\in\mathcal{V}^{h,u}_{2(\log n)^2,\lambda_n}\}}\big(N_z^{\tau^{p_n}}\big)^2\Big].
\end{align*}
One can notice that, if $u\in\mathcal{V}^{h}_{(\log n)^2,\lambda_n}$ and $z>u$ such that $z\in\mathcal{V}^{h,u}_{2(\log n)^2,\lambda_n}$, then for any $u<w\leq z$, $H_w=(H_u-1)e^{-V_u(w)}+H_{u,w}\leq(1+e^h)\lambda_n$ and $V(w)=V_u(w)+V(u)\geq -2h$ so $z\in\mathcal{V}^{2h}_{2(\log n)^2,(1+e^h)\lambda_n}$. Hence, $\Tilde{\Sigma}_n\leq 3c_1\E[\sum_{|z|=2(\log n)^2-1}\un_{\{z\in\mathcal{V}^{2h}_{2(\log n)^2,(1+e^h)\lambda_n}\}}(N_z^{\tau^{p_n}})^2]$ and similar arguments as the ones used to prove $\Sigma_n=o(n^{2r})$ lead to $\Tilde{\Sigma}_n=o(n^{2r})$. Exact same arguments yield
\begin{align*}
    \E\Big[\sum_{|u|=(\log n)^2}\un_{\{u\in\mathcal{V}^{h}_{(\log n)^2,\lambda_n}\}}\sum_{x\not=y\in\T;\; x\land y=u}\un_{\{x,y\in\Tilde{\mathcal{B}}^{p_n}_u\}}\Big]=o(n^{2r}),
\end{align*}
thus proving that $\E[\sum_{|u|=(\log n)^2}(\Tilde{B}^{p_n}_{u,\lambda_n})^2\un_{\{u\in\mathcal{V}^h_{(\log n)^2,\lambda_n}\}}]=o(n^{2r})$ as $n\to\infty$. \\
We are left to deal with $\E[\sum_{|u|=|v|=(\log n)^2}(N_u^{\tau{p_n}}-\Tilde{B}^{p_n}_{u,\lambda_n})N_v^{\tau{p_n}}\un_{\{u,v\in\mathcal{V}^h_{(\log n)^2,\lambda_n}\}}\un_{\{N_u^{\tau^{p_n}}\leq hp_n\}}]$. Thanks to the mixed branching property \eqref{BP} and Fact \ref{CardSetRoot}
\begin{align*}
    &\E\Big[\sum_{|u|=|v|=(\log n)^2}\big(N_u^{\tau{p_n}}-\Tilde{B}^{p_n}_{u,\lambda_n}\big)N_v^{\tau{p_n}}\un_{\{u,v\in\mathcal{V}^h_{(\log n)^2,\lambda_n}\}}\un_{\{N_u^{\tau^{p_n}}\leq hp_n\}}\Big] \\ & =\E\Big[\sum_{|u|=|v|=(\log n)^2}\Bar{\Psi}^h_n\big(N_u^{\tau{p_n}}\big)N_v^{\tau{p_n}}\un_{\{u,v\in\mathcal{V}^h_{(\log n)^2,\lambda_n}\}}\Big],
\end{align*}
where we recall that for any $m,h\in\N^*$ and $h>0$, $\Bar{\Psi}^h_n(m)=(m-\Psi^h_n(m))\un_{\{m\leq hp_n\}}$ and $\Psi^h_n(m)=\E[\sum_{x\in\mathcal{V}^h_{(\log n)^2,\lambda_n}}\un_{\{N_x^{\tau^{m}}=1,\; \min_{1\leq i<|x|}N_{x_i}^{\tau^{m}}\geq 2\}}]$. Also recall that for any $\zeta\in(0,1-1/\kappa)$, $0\leq\Bar{\Psi}^h_n(m)\leq C_{\ref{CardSetRoot}}m\P(\mathcal{R}_{\tau^{hp_n},(\log n)^2}\not\subset\mathcal{V}^{h}_{(\log n)^2,\lambda_n})^{\zeta}$, thus giving
\begin{align*}
    &\E\Big[\sum_{|u|=|v|=(\log n)^2}\big(N_u^{\tau{p_n}}-\Tilde{B}^{p_n}_{u,\lambda_n}\big)N_v^{\tau{p_n}}\un_{\{u,v\in\mathcal{V}^h_{(\log n)^2,\lambda_n}\}}\un_{\{N_u^{\tau^{p_n}}\leq hp_n\}}\Big] \\ & \leq C_{\ref{CardSetRoot}}\E\Big[\sum_{|u|=|v|=(\log n)^2}N_u^{\tau{p_n}}N_v^{\tau{p_n}}\un_{\{u,v\in\mathcal{V}^h_{(\log n)^2,\lambda_n}\}}\Big]\times\P(\mathcal{R}_{\tau^{hp_n},(\log n)^2}\not\subset\mathcal{V}^{h}_{(\log n)^2,\lambda_n})^{\zeta},
\end{align*}
and with similar arguments as the ones used to show that $\Sigma_n=o(n^{2r})$ as $n\to\infty$, one can prove that $\E[\sum_{|u|=|v|=(\log n)^2}N_u^{\tau{p_n}}N_v^{\tau{p_n}}\un_{\{u,v\in\mathcal{V}^h_{(\log n)^2,\lambda_n}\}}]=O(n^{2r})$. Hence, by Lemma \ref{RegLine} first and by \eqref{MoyenneCarre}, using the fact that $\E[\sum_{|u|=(\log n)^2}(\Tilde{B}^{p_n}_{u,\lambda_n})^2\un_{\{u\in\mathcal{V}^h_{(\log n)^2,\lambda_n}\}}]=o(n^{2r})$, we obtain

\begin{align}\label{EspPart12}
    \limsup_{n\to\infty}\E\Big[|\sum_{|u|=(\log n)^2}(\Tilde{B}^{p_n}_{u,\lambda_n}-N_u^{\tau{p_n}})\un_{\{u\in\mathcal{V}^h_{(\log n)^2,\lambda_n}\}}\un_{\{N_u^{\tau^{p_n}}\leq hp_n\}}|^2\Big]\leq C_{\ref{RootNumber},2}e^{-\eta_0h\zeta},
\end{align}
for some constant $C_{\ref{RootNumber},2}>0$. By \eqref{DecoupeEsp}, since \eqref{EspPart11} and \eqref{EspPart12} hold for any $h>0$, \eqref{EspPart1} is proved by taking $h\to+\infty$ and this ends the first step.

\vspace{0.2cm}

\noindent The second step is dedicated to the proof of the following: for any $\varepsilon_1>0$
\begin{align}\label{EspPart2}
    \lim_{n\to\infty}\P^*\Big(\Big|\frac{1}{n^r}B^{p_n}_{(\log n)^2}-\alpha W_{\infty}\Big|>\varepsilon_1\Big)=0.
\end{align}
For that, we use the first step, telling, that somehow, $B^{p_n}_{(\log n)^2}$ behaves like $\sum_{|u|=(\log n)^2}N_u^{\tau^{p_n}}$ and we then prove that $\sum_{|u|=(\log n)^2}N_u^{\tau^{p_n}}$ concentrates around its quenched mean. \\ 
One can notice that $B^{p_n}_{(\log n)^2}=\sum_{|u|=(\log n)^2}\Tilde{B}^{p_n}_{u}$ where we recall that for any $u\in\T$, $\Tilde{B}^{p_n}_u$ is equal to $\sum_{x>u}\un_{\{N_x^{\tau^{p_n}}=1,\; \min_{|u|<i<|x|}N_{x_i}^{\tau^{p_n}}\geq 2\}}$ so
\begin{align*}
    \P^*\Big(\Big|\frac{1}{n^r}B^{p_n}_{(\log n)^2}-\alpha W_{\infty}\Big|>\varepsilon_1\Big)\leq&\P^*\Big(\Big|\sum_{|u|=(\log n)^2}\un_{\{u\in\mathcal{V}^h_{(\log n)^2,\lambda_n}\}}\un_{\{N_u^{\tau^{p_n}}\leq hp_n\}}(\Tilde{B}^{p_n}_u-N_u^{\tau^{p_n}})\Big|>\varepsilon_1n^r/2\Big) \\ & +\P^*\Big(\Big|\sum_{|u|=(\log n)^2}\un_{\{u\in\mathcal{O}_{(\log n)^2,\lambda_n}\}}\big(N_u^{\tau^{p_n}}-\E^{\mathcal{E}}[N_u^{\tau^{p_n}}]\big)\Big|>\varepsilon_1n^r/4\Big) \\ & +\Pb^*\Big(\Big|\frac{1}{n^r}\E^{\mathcal{E}}\Big[\sum_{|u|=(\log n)^2}N_u^{\tau^{p_n}}\Big]-\alpha W_{\infty}\Big|>\varepsilon_1/4\Big) \\ & +2\Big(1-\P^*\big(\Tilde{V}^{h,p_n}_{(\log n)^2,\lambda_n}\big)\Big),
\end{align*}
where $\Tilde{V}^{h,p_n}_{(\log n)^2,\lambda_n}$ is defined to be the event $\{\mathcal{R}_{\tau^{p_n}, (\log n)^2}\subset\mathcal{V}^{h}_{(\log n)^2,\lambda_n},\; \max_{|u|=(\log n)^2}N_u^{\tau^{p_n}}\leq hp_n\}$ and we recall that $\lambda_n=n^{r(1+1/\kappa)/2}$. Let us show that $\lim_{n\to\infty}\P^*(|\sum_{|u|=(\log n)^2}\un_{\{u\in\mathcal{O}_{(\log n)^2,\lambda_n}\}}(N_u^{\tau^{p_n}}-\E^{\mathcal{E}}[N_u^{\tau^{p_n}}])|>\varepsilon_1n^r/4)=0$. Indeed, restricting ourselves to the set of vertices $\mathcal{O}_{(\log n)^2,\lambda_n}$, we have, using the Bienaymé-Tchebychev inequality
\begin{align*}
    &\P^*\Big(\Big|\sum_{|u|=(\log n)^2}\un_{\{u\in\mathcal{O}_{(\log n)^2,\lambda_n}\}}\big(N_u^{\tau^{p_n}}-\E^{\mathcal{E}}[N_u^{\tau^{p_n}}]\big)\Big|>\varepsilon_1n^r/4\Big) \\ & \leq\frac{16\alpha}{\varepsilon_1^2n^r}\E^{\mathcal{E}}\Big[\Big(\sum_{|u|=(\log n)^2}\un_{\{u\in\mathcal{O}_{(\log n)^2,\lambda_n}\}}N_u^{\tau^{1}}\Big)^2\Big],
\end{align*}
where we have used that the random variables $\sum_{|u|=(\log n)^2}\un_{\{u\in\mathcal{O}_{(\log n)^2,\lambda_n}\}}(N_u^{\tau^j}-N_u^{\tau^{j-1}})$, $j\in\{1,\ldots,p_n\}$, are independent and distributed as $\sum_{|u|=(\log n)^2}\un_{\{u\in\mathcal{O}_{(\log n)^2,\lambda_n}\}}N_u^{\tau^1}$ under $\P^{\mathcal{E}}$, thanks to the strong Markov property. Then, thanks to Lemma \ref{LemmaEdge}
\begin{align*}
    \E^{\mathcal{E}}\Big[\Big(\sum_{|u|=(\log n)^2}\un_{\{u\in\mathcal{O}_{(\log n)^2,\lambda_n}\}}N_u^{\tau^{1}}\Big)^2\Big]\leq &2\sum_{|u|=(\log n)^2}\un_{\{u\in\mathcal{O}_{(\log n)^2,\lambda_n}\}}H_ue^{-V(u)} \\ & +2\underset{|u|=|v|=(\log n)^2}{\sum_{u\not=v}}\un_{\{u\land v\in\mathcal{O}_{(\log n)^2,\lambda_n}\}}H_{u\land v}e^{V(u\land v)}e^{-V(u)-V(v)}.
\end{align*}
By definition of $\mathcal{O}_{(\log n)^2,\lambda_n}$, $\un_{\{w\in\mathcal{O}_{(\log n)^2,\lambda_n}\}}H_{w}\leq\lambda_n$ so $\Eb[\sum_{|u|=(\log n)^2}\un_{\{u\in\mathcal{O}_{(\log n)^2,\lambda_n}\}}H_ue^{-V(u)}]\leq\lambda_n$ and $\Eb[\sum_{u\not=v;|u|=|v|=(\log n)^2}\un_{\{u\land v\in\mathcal{O}_{(\log n)^2,\lambda_n}\}}H_{u\land v}e^{V(u\land v)}e^{-V(u)-V(v)}]\leq c_1\lambda_n(\log n)^2$ thus implying that the annealed mean $\E[(\sum_{|u|=(\log n)^2}\un_{\{u\in\mathcal{O}_{(\log n)^2,\lambda_n}\}}N_u^{\tau^{1}})^2]$ is smaller than $2(1+c_1)\lambda_n(\log n)^2=2(1+c_1)n^{r(1+1/\kappa)/2}(\log n)^2$ and since $r(1+1/\kappa)/2<r$, we obtain
\begin{align*}
    \lim_{n\to\infty}\P^*\Big(\Big|\sum_{|u|=(\log n)^2}\un_{\{u\in\mathcal{O}_{(\log n)^2,\lambda_n}\}}\big(N_u^{\tau^{p_n}}-\E^{\mathcal{E}}[N_u^{\tau^{p_n}}]\big)\Big|>\varepsilon_1n^r/4\Big)=0.
\end{align*}
We now turn to $\Pb^*(|\frac{1}{n^r}\E^{\mathcal{E}}[\sum_{|u|=(\log n)^2}N_u^{\tau^{p_n}}]-\alpha W_{\infty}|>\varepsilon_1/4)$. As we have already mentioned previously, $\E^{\mathcal{E}}[\sum_{|u|=(\log n)^2}N_u^{\tau^{p_n}}]=p_n\sum_{|x|=(\log n)^2}e^{-V(x)}$ so $\E^{\mathcal{E}}[\sum_{|u|=(\log n)^2}N_u^{\tau^{p_n}}]/n^r$ goes to $\alpha W_{\infty}$ $\Pb^*$-almost surely as $n\to\infty$. Finally, note that, thanks to Lemma \ref{RegLine} and Markov inequality
\begin{align*}
    \limsup_{n\to\infty}\Big(1-\P^*\big(\Tilde{V}^{h,p_n}_{(\log n)^2,\lambda_n}\big)\Big)&\leq\limsup_{n\to\infty}\Big(\P^*\big(\mathcal{R}_{\tau^{p_n}, (\log n)^2}\not\subset\mathcal{V}^{h}_{(\log n)^2,\lambda_n}\big)+\frac{1}{hp_n}\E\Big[\sum_{|u|=(\log n)^2}N_u^{\tau^{p_n}}\Big]\Big) \\ & \leq \frac{e^{-\eta_0h}}{1-e^{-|\psi(\eta_0)|}}+\frac{1}{h},
\end{align*}
so thanks to \eqref{EspPart1}, letting $h\to\infty$ leads to \eqref{EspPart2}.

\vspace{0.2cm}

\noindent In the third and last step, we show \eqref{EspPart2} actually holds uniformly in $\alpha\in[0,M]$. Indeed, on the event $\{(\mathcal{B}^{p}_{(\log n)^2})_{p\leq \lfloor Mn^{r}\rfloor}\textrm{ is non-decreasing}\}$, the function $\alpha\in[0,M]\mapsto\frac{1}{n^r}B^{\lfloor\alpha n^r\rfloor}_{(\log n)^2}$ is non-decreasing for any $n\in\N^*$. Moreover, the function $\alpha\in[0,M]\mapsto\alpha W_{\infty}$ is continuous so Lemma \ref{NDSet}, together with Dini's theorem leads to the result and this ends the proof. 
\end{proof}

\subsection{Proofs of Theorem \ref{Th2} to Theorem \ref{Th3} and Corollary \ref{Coro1}}

\noindent Recall that $T^0=0$, for any $j\in\N^*$, $T^j=\inf\{k>T^{j-1};\; X_k=e^*\}$ and $\mathcal{L}^m=\sum_{k=1}^m\un_{\{X_k=e^*\}}$. We are now ready to prove our three theorems and our corollary. We proceed as follows: we first introduce a new random walk we use to prove that for all $\kappa>1$, the sequences of processes $((T^{\lfloor\alpha n\rfloor}/\kappa_n;\; \alpha\geq 0))_{n\in\N^*}$ and $((\mathcal{L}^{\lfloor t\kappa_n\rfloor}/n;\; t\geq 0))_{n\in\N^*}$ converge in law, under the annealed probability $\P^*$, for the Skorokhod topology on $D([0,\infty))$, recalling that $\kappa_n=n^{\kappa\land 2}$ if $\kappa\not=2$ and $\kappa_n=n^2/\log n$ if $\kappa=2$. From this, we deduce the joint convergence in law, under $\P^*$, of $((T^{\lfloor\alpha n\rfloor}/\kappa_n;\; \alpha\geq 0),(\mathcal{L}^{\lfloor t\kappa_n\rfloor}/n;\; t\geq 0))_{n\in\N^*}$ for the Skorokhod product topology on $D([0,\infty))\times D([0,\infty))$. Then, we show that this joint converge holds in law under the quenched probability $\P^{\mathcal{E}}$, for $\Pb^*$-almost every environment, thus proving Theorem \ref{Th2} and Theorem \ref{Th1}. Finally, the end of this section will be dedicated to the proofs of Theorem \ref{Th3} and Corollary \ref{Coro1}. Note that it could be possible to prove directly the joint convergence under the annealed probability $\P^*$ but we choose to treat each component separately in order to avoid unnecessary complication.

\begin{proof}[Proofs of Theorem \ref{Th2} and Theorem \ref{Th1}]\phantom{.}

\vspace{0.3cm}

\noindent Let $\ell_n:=(\log n)^2$, recall that $\mathcal{B}^{p}_{\ell_n}$ is the set $\{x\in\T,\; |x|>\ell_n;\; N_x^{\tau^p}=1\textrm{ and }\min_{\ell<i<|x|}N_{x_i}^{\tau^{p}}\geq 2\}$ and $B^{p}_{\ell_n}$ is the cardinal of $\mathcal{B}^{p}_{\ell_n}$. On the event $\{\mathcal{B}^{p}_{\ell_n}\not=\varnothing\}$, for any $i\in\{1,\ldots,B^{p}_{\ell_n}\}$, denote by $e_i^{p}$ the $i$-th vertex of $\mathcal{B}^{p}_{\ell_n}$ visited by the random walk $(X_j)_{j\in\N}$. Also define the tree $\T^p_i$ rooted at $e^p_i$ by $\T^p_i:=\{x\in\T;\; x\geq e^p_i\}$. $\T_i^{p}$ is then the sub-tree of $\T$ made up of the descendants of $e_i^{p}$. \\
Introduce $T_{e_i^{p}}:=\inf\{k\geq 1;\; X_k=e_i^{p}\}$, define the translated walk $X^{p}_{i,j}:=X_{j+T_{e_i^{p}}}$ starting from $e_i^{p}$ under $\P^{\mathcal{E}}$. Also define $\tau^p_0=0$, for $i\geq 1$, let $\tau^{p}_i:=\inf\{j\geq 1;\; X^{p}_{i,j}=(e^{p}_i)^*\}$ and $\Tilde{\tau}^p_i:=\sum_{\mathfrak{l}=0}^i\tau^p_{\mathfrak{l}}$.Then, introduce the random walk $\Tilde{X}^p$, defined by: for all $j\in\N$
\begin{align*}
    &\Tilde{X}^p_j=X^{p}_{i,j-\Tilde{\tau}^p_{i-1}}\;\;\textrm{ if and only if}\;\; j\in\big\{\Tilde{\tau}^p_{i-1},\ldots,\Tilde{\tau}^p_i-1\big\}\;\textrm{ for some }\; i\in\{1,\ldots,B^p_{\ell_n}\} \\ & \textrm{and } X^{p}_j=e^p_{B^p_{\ell_n}}\;\textrm{ for any }\; j\geq\Tilde{\tau}^p_{B^p_{\ell_n}},
\end{align*}
Hence, under $\P^{\mathcal{E}}$, $\Tilde{X}^p$ is a $\bigcup_{i=1}^{B^p_{\ell_n}}\T^p_i$-valued random walk starting from $e^p_1$ and absorbed in $e^p_{B^p_{\ell_n}}$ such that, for any $i\in\{1,\ldots,B^p_{\ell_n}\}$ and any $j\in\{\Tilde{\tau}^p_{i-1},\ldots,\Tilde{\tau}^p_i-1\}$, $\Tilde{X}^p_j$ belongs to the tree $\T^p_i$. Moreover, it behaves similarly as the restriction of the original random walk $\X$ to the set $\bigcup_{i=1}^{B^p_{\ell_n}}\T^p_i$. \\
Let us now define the range associated with the random walk $\Tilde{X}^p$. For that, introduce the following edge local time: for any $i\in\{1,\ldots,B^p_{\ell_n}\}$, let $N^p_{i,e^p_i}=1$ and for any $x\in\T^p_i$, $x\not=e^p_i$
\begin{align*}
    N_{i,x}^{p}:=\sum_{k=\Tilde{\tau}^p_{i-1}+1}^{\Tilde{\tau}^{p}_i-1}\un_{\{\Tilde{X}^{p}_{k-1}=x^*,\; \Tilde{X}^{p}_{k}=x\}},    
\end{align*}
the number of times the oriented edge $(x^*,x)$ has been visited by $(\Tilde{X}^p_j)_{j\in\{\Tilde{\tau}^p_{i-1}+1,\ldots,\Tilde{\tau}^p_i-1\}}$. Note, by definition, that $N^p_{i,x}$ coincides with the regular edge local time $N^{\tau^p}_x$. We can then define the tree $\mathcal{R}_i^{p}$ rooted at $e_i^{p}$ by 
\begin{align*}
    \mathcal{R}_i^{p}:=\{x\geq e_i^{p};\; N^p_{i,x}\geq 1\}.
\end{align*}
$\mathcal{R}_i^{p}$ is the sub-tree of $\T^{p}_i$ made up of the vertices visited by $(\Tilde{X}^p_j)_{j\in\{\Tilde{\tau}^p_{i-1}+1,\ldots,\Tilde{\tau}^p_i-1\}}$. Note that, by definition, we have $\mathcal{R}_i^{p}\cap\mathcal{R}_{i'}^{p}=\varnothing$ when $i\not=i'$. Besides, if $\Tilde{\tau}^p_i-1=\Tilde{\tau}^p_{i-1}$, then $\mathcal{R}^p_i=\{e^p_i\}$. \\
For any $i\in\N^*$, define the multi-type tree $(\Tilde{\mathcal{R}}^{p}_i,(\Tilde{N}^{p}_{i,x};\; x\in\Tilde{\mathcal{R}}^{p}_{i,x}))$ by: for all $i\in\{1,\ldots,B^{p}_{\ell_n}\}$, $\Tilde{\mathcal{R}}^{p}_i=\mathcal{R}^{p}_i$, $\Tilde{N}^{p}_{i,x}=N^{p}_{i,x}$ for any $x\in\mathcal{R}^{p}_i$ and $(\Tilde{\mathcal{R}}^{p}_i,(\Tilde{N}^{p}_{i,x};\; x\in\Tilde{\mathcal{R}}^{p}_{i,x}))$, $i>B^{p}_{\ell_n}$, are independent copies of $(\mathcal{R}_{\tau^1},(N_x^{\tau^1};\; x\in\mathcal{R}_{\tau^1}))$. Moreover, the family $((\Tilde{\mathcal{R}}^{p}_i,(\Tilde{N}^{p}_{i,x};\; x\in\Tilde{\mathcal{R}}^{p}_{i,x})))_{i>B^{p}_{\ell_n}}$ is assumed to be independent of $\Tilde{X}^{p}$. Therefore, under $\P$, $((\Tilde{\mathcal{R}}^{p}_i,(\Tilde{N}^{p}_{i,x};\; x\in\Tilde{\mathcal{R}}^{p}_{i,x})))_{i\in\N^*}$ is distributed as the multi-type Galton-Watson forest $\mathcal{F}^1=((\mathcal{R}_{1,i},(\beta_i(x);\; x\in\mathcal{R}_{1,i})))_{i\in\N^*}$ which satisfies hypothesis \ref{H1}, \ref{H2} and \ref{H3}, where we recall that $(\mathcal{R}_{1,i},(\beta_i(x);\; x\in\mathcal{R}_{1,i}))$, $i\in\N^*$, are independent copies of $(\mathcal{R}_{\tau^1},(N_x^{\tau^1}; x\in\mathcal{R}_{\tau^1}))$ under $\P$, see Fact \ref{GWF_Hyp}.

\vspace{0.3cm}

\noindent \textbf{Convergence of $((T^{\lfloor\alpha n\rfloor}/\kappa_n;\; \alpha\geq 0))_{n\in\N^*}$ under $\P^*$}

\vspace{0.2cm}

\noindent One can notice that it is enough to prove the convergence of the sequence $((T^{\lfloor\alpha n\rfloor}/\kappa_n;\; \alpha\in[\delta,M]))_{n\in\N}$ for any $M>0$ and $\delta\in(0,M)$. Indeed, assume the convergence holds on $[\delta,M]$ for any $\delta\in(0,M)$. Then, the convergence in law of the finite-dimensional of $((T^{\lfloor\alpha n\rfloor}/\kappa_n;\; \alpha\in[0,M]))_{n\in\N^*}$ follows immediately. Moreover, since the function $\alpha\in[0,M]\mapsto T^{\lfloor\alpha n\rfloor}$ is non-decreasing, we have that $\P^{\mathcal{E}}(\sup_{\alpha\in[0,\delta)}T^{\lfloor\alpha n\rfloor}>\varepsilon_1\kappa_n)$ is smaller than $\P^{\mathcal{E}}(T^{\lfloor\delta n\rfloor}>\varepsilon_1\kappa_n)$ and then, thanks to the scaling property of $(\tau_{\alpha}(\VectCoord{\Tilde{Y}}{\kappa});\; \alpha)$, see \eqref{LaplaceTA}, $\lim_{\delta\to 0}\limsup_{n\to\infty}\P^{\mathcal{E}}(\sup_{\alpha\in[0,\delta)}T^{\lfloor\alpha n\rfloor}>\varepsilon_1\kappa_n)=\lim_{\delta\to 0}\mathtt{P}(\tau_{1}(\VectCoord{\Tilde{Y}}{\kappa})>\varepsilon_1/\delta^{\kappa\land 2})=0$ for any $\varepsilon_1>0$,
which yields the tightness of the sequence $((T^{\lfloor\alpha n\rfloor}/\kappa_n;\; \alpha\in[0,M]))_{n\in\N^*}$. \\
Also recall that $\tau^0=0$ and for any $j\in\N^*$, $\tau^j=\inf\{k>T^{j-1};\; X_k=e,\; X_{k-1}=e^*\}$. One can see that, under $\P^{\mathcal{E}}$, for any $p\in\N^*$, $T^p=\tau^p-p$. Hence, we will prove the convergence of $((\tau^{\lfloor\alpha n\rfloor}/\kappa_n;\; \alpha\in[0,M]))_{n\in\N^*}$ for any $\delta\in(0,M)$ instead.

\vspace{0.2cm}

\noindent Inspired from \cite{AidRap} and \cite{deRaph1}, we relate $\tau^{\lfloor\alpha n\rfloor}$ with the types of the multi-type Galton-Watson tree $(\mathcal{R}_{\tau^{\lfloor\alpha n\rfloor}},(N_x^{\tau^{\lfloor\alpha n\rfloor}};\; x\in\mathcal{R}_{\tau^{\lfloor\alpha n\rfloor}}))$, see Fact \ref{FactGWMulti}. For that, introduce, for any $i\in\{1,\ldots,B^p_{\ell_n}\}$ and any $x\in\T^p_i$ the local time $\mathcal{L}^p_{i,x}:=N^p_{i,x}+\sum_{y;\; y^*=x}N^p_{i,y}$. Let us show that for any $\varepsilon_1>0$ and all $\kappa>1$, $\Pb^*$-almost surely
\begin{align}\label{ProbaTA1}
    \lim_{n\to\infty}\P^{\mathcal{E}}\Big(\sup_{p\in[\lfloor\delta n\rfloor,\ldots,\lfloor Mn\rfloor]}\Big|\tau^{p}-\sum_{i=1}^{B^{p}_{\ell_n}}\sum_{x\in\mathcal{R}^{p}_i}\mathcal{L}_{i,x}^{p}\Big|>\varepsilon_1\kappa_n\Big)=0,
\end{align}
with the convention $\sum_{\varnothing}=0$. Indeed, as we previously noticed, $\mathcal{L}^p_{i,x}=\mathcal{L}^{\tau^p}_x=\sum_{j=1}^{\tau^p}\un_{\{X_j=x\}}$ for any $x\in\T^p_i$ and since $\tau^{p}=\sum_{x\in\mathcal{R}_{\tau^{p}}}\mathcal{L}_x^{\tau^{p}}$, we have
\begin{align*}
    \tau^{p}-\sum_{i=1}^{B^{p}_{\ell_n}}\sum_{x\in\mathcal{R}^{p}_i}\mathcal{L}_{i,x}^{p}=\sum_{x\in C^{p}_{\ell_n}}\mathcal{L}_x^{\tau^p},
\end{align*}
where $x\in C^{p}_{\ell_n}$ if and only if $x\in\mathcal{R}_{\tau^p}$ and for all $i\in\{1,\ldots,B^{p}_{\ell_n}\}$, $e_i^{p}$ is not an ancestor of the vertex $x$. Then, on the event $\bigcap_{p\in\{\lfloor\delta n\rfloor,\ldots,\lfloor Mn\rfloor\}}\{\sup_{x\in C^{p}_{\ell_n}}|x|\leq(\log n)^3\}=\{\sup_{x\in C^{\lfloor \delta n\rfloor}_{\ell_n}}|x|\leq(\log n)^3\}$
\begin{align*}
    \sup_{p\in[\lfloor\delta n\rfloor,\ldots,\lfloor Mn\rfloor]}\Big|\tau^{p}-\sum_{i=1}^{B^{p}_{\ell_n}}\sum_{x\in\mathcal{R}^{p}_i}\mathcal{L}_{i,x}^{p}\Big|\leq\sum_{|x|\leq(\log n)^3}\mathcal{L}_x^{\tau^{\lfloor Mn\rfloor}}.
\end{align*}
Moreover, $\E^{\mathcal{E}}[\sum_{|x|\leq(\log n)^3}\mathcal{L}_x^{\tau^{\lfloor Mn\rfloor}}]=\lfloor Mn\rfloor\sum_{k\leq(\log n)^3}(W_{k}+W_{k+1})$, where we recall that $W_{k}$ equals $\sum_{|x|=k}e^{-V(x)}$. Hence, thanks to the Markov inequality
\begin{align*}
    &\P^{\mathcal{E}}\Big(\sup_{p\in[\lfloor\delta n\rfloor,\ldots,\lfloor Mn\rfloor]}\Big|\tau^{p}-\sum_{i=1}^{B^{p}_{\ell_n}}\sum_{x\in\mathcal{R}^{p}_i}\mathcal{L}_{i,x}^{p}\Big|>\varepsilon_1\kappa_n\Big) \\ & \leq\P^{\mathcal{E}}\Big(\sup_{x\in C^{\lfloor \delta n\rfloor}_{\ell_n}}|x|>(\log n)^3\Big)+\P^{\mathcal{E}}\Big(\sum_{|x|\leq(\log n)^3}\mathcal{L}_x^{\tau^{\lfloor Mn\rfloor}}>\varepsilon_1\kappa_n\Big) \\ & \leq \P^{\mathcal{E}}\Big(\sup_{x\in C^{\lfloor \delta n\rfloor}_{\ell_n}}|x|>(\log n)^3\Big)+M\frac{n}{\kappa_n}\sum_{k\leq(\log n)^3}\big(W_k+W_{k+1}\big).
\end{align*}
Note that, $\Pb$-almost surely, $\lim_{n\to\infty}\P^{\mathcal{E}}(\sup_{x\in C^{\lfloor\delta n\rfloor}_{\ell_n}}|x|>(\log n)^3)=0$ (the proof is similar to the one of Lemma \ref{PropSetRoots1}). Moreover, recall that the martingale $(W_{k})_k$ converges $\Pb^*$-almost surely to the positive random variable $W_{\infty}$ so does $((\log n)^{-3}\sum_{k\leq(\log n)^3}W_k)_n$ and since we have $\kappa>1$, $(\frac{n}{\kappa_n}\sum_{k\leq(\log n)^3}\big(W_k+W_{k+1}))$ goes to $0$, $\Pb^*$-almost surely as $n\to\infty$, which yields \eqref{ProbaTA1}. 

\vspace{0.2cm}
    
\noindent Let us now prove that $((\frac{1}{\kappa_n}\sum_{i=1}^{B^{\lfloor\alpha n\rfloor}_{\ell_n}}\sum_{x\in\mathcal{R}^{\lfloor\alpha n\rfloor}_i}\mathcal{L}_{i,x}^{\lfloor\alpha n\rfloor}; \alpha\in[\delta,M]))_{n\in\N^*}$ converges in law under $\P^*$ to the limit in Theorem \ref{Th1}. One can notice, on the event $\{(\mathcal{B}^{p}_{\ell_n})_{p\leq \lfloor Mn\rfloor}\textrm{ is non-decreasing};\; \mathcal{B}^{\lfloor \delta n\rfloor}_{\ell_n}\not=\varnothing\}$ (whose quenched probability goes to $1$ as $n$ goes to $\infty$ by Lemma \ref{NDSet} and Lemma \ref{PropSetRoots1} \ref{Fact2.1}), that for all $\alpha\in[\delta,M]$ and any $i\in\{1,\ldots,B^{\lfloor\alpha n\rfloor}_{\ell_n}\}$, we have $e^{\lfloor\alpha n\rfloor}_i=e^{\lfloor Mn\rfloor}_i$ and $N^{\lfloor\alpha n\rfloor}_{i,z}=N^{\lfloor Mn\rfloor}_{i,z}$ for any $z\in\T^{\lfloor\alpha n\rfloor}_i$ thus giving $\mathcal{R}^{\lfloor\alpha n\rfloor}_i=\mathcal{R}^{\lfloor Mn\rfloor}_i$ and then, denoting $\Tilde{\mathcal{L}}^p_{i,x}=\Tilde{N}^p_{i,x}+\sum_{y;\; y^*=x}\Tilde{N}^p_{i,x}$, we have
\begin{align*}
    \sum_{i=1}^{B^{\lfloor\alpha n\rfloor}_{\ell_n}}\sum_{x\in\mathcal{R}^{\lfloor\alpha n\rfloor}_i}\mathcal{L}_{i,x}^{\lfloor\alpha n\rfloor}=\sum_{i=1}^{B^{\lfloor\alpha n\rfloor}_{\ell_n}}\sum_{x\in\mathcal{R}^{\lfloor Mn\rfloor}_i}\mathcal{L}_{i,x}^{\lfloor Mn\rfloor}=\sum_{i=1}^{B^{\lfloor\alpha n\rfloor}_{\ell_n}}\sum_{x\in\Tilde{\mathcal{R}}^{\lfloor Mn\rfloor}_i}\Tilde{\mathcal{L}}_{i,x}^{\lfloor Mn\rfloor}.
\end{align*}
Moreover, under $\P$, $((\Tilde{\mathcal{R}}^{\lfloor Mn\rfloor}_i,(\Tilde{N}^{\lfloor Mn\rfloor}_{i,x};\; x\in\Tilde{\mathcal{R}}^{\lfloor Mn\rfloor}_{i,x})))_{i\in\N^*}$ is distributed as the multi-type Galton-Watson forest $\mathcal{F}^1=((\mathcal{R}_{1,i},(\beta_i(x);\; x\in\mathcal{R}_{1,i})))_{i\in\N^*}$ which satisfies hypothesis \ref{H1}, \ref{H2} and \ref{H3} (see Fact \ref{GWF_Hyp}), implying that the process $(\frac{1}{\kappa_n}\sum_{i=1}^{\lfloor\alpha' n\rfloor}\sum_{x\in\Tilde{\mathcal{R}}^{\lfloor Mn\rfloor}_i}\Tilde{\mathcal{L}}_{i,x}^{\lfloor Mn\rfloor};\; \alpha'\geq 0)$ is distributed as $(\frac{1}{\kappa_n}F_{\lfloor\alpha' n\rfloor};\; \alpha'\geq 0)$ where $F_p:=\sum_{i=1}^{p}\sum_{x\in\mathcal{R}_{1,i}}(\beta_i(x)+\sum_{y;\; y^*=x}\beta_i(y))$, with the convention $F_0=0$. Hence, by Proposition \ref{PropArbresMulti}, together with Lemma \ref{RootNumber}, stating that 
\begin{align*}
\lim_{n\to\infty} \P^*\Big(\sup_{\alpha\in[0,M]}\Big|\frac{1}{n}B^{\lfloor\alpha n\rfloor}_{\ell_n}-\alpha W_{\infty}\Big|>\varepsilon_1\Big)=0,
\end{align*}
and the scaling property of $(\tau_{\alpha}(\VectCoord{\Tilde{Y}}{\kappa});\; \alpha\geq 0)$ (see \eqref{LaplaceTA}), we finally get the convergence of $(\frac{1}{\kappa_n}F_{B^{\lfloor\alpha n\rfloor}_{\ell_n}};\; \alpha\in[\delta,M])$ in law, under $\P^*$, towards the limit in Theorem \ref{Th1}, thus giving the wanted convergence.

\vspace{0.3cm}

\noindent\textbf{Convergence of $((\mathcal{L}^{\lfloor t\kappa_n\rfloor}/n;\; t\geq 0))_{n\in\N^*}$ under $\P^*$}

\vspace{0.2cm}

\noindent As we said before, it is enough to prove the convergence of $((\mathcal{L}^{\lfloor t\kappa_n\rfloor}/n;\; t\in[\delta',M']))_{n\in\N^*}$ for any $M'>0$ and all $\delta'\in(0,M')$. Moreover, note that under $\P^{\mathcal{E}}$, $\mathcal{L}^{m}=N^{m+1}$ where $N^{m+1}:=N^{m+1}_{e^*}=\sum_{k=1}^{m+1}\un_{\{X_{k-1}=e^*,\; X_k=e\}}$, so we will deal with the sequence of processes $((N^{\lfloor t\kappa_n\rfloor}/n;\; t\in[\delta',M']))_{n\in\N^*}$. 

\vspace{0.2cm}

\noindent We claim that the total amount of time spent by the random $\X$ on the oriented edge $(e^*,e)$ up to the time $\lfloor t\kappa_n\rfloor$ behaves like the number of vertices in $\mathcal{B}^{n^r}_{\ell_n}$ visited by the random walk $\Tilde{X}^{n^r}$ before $\lfloor t\kappa_n\rfloor$ for some $r>0$, or, equivalently, behaves like the index of the tree in the forest $\mathcal{R}^{n^r}$ to which the vertex $\Tilde{X}^{n^r}_{\lfloor t\kappa_n\rfloor}$ belongs. Indeed, introduce, for any $x\in\bigcup_{i=1}^{B^p_{\ell_n}}\T^p_i$, $\Tilde{T}^p_x:=\inf\{j\geq 0;\; \Tilde{X}^p_j=x\}$, let $m\in\N^*$ and denote by $\Tilde{R}^{n^r}_m:=\sum_{x\in\mathcal{B}^{n^r}_{\ell_n}}\un_{\{\Tilde{T}^{n^r}_x\leq m\}}$ the number of distinct elements of $\mathcal{B}^{n^r}_{\ell_n}$ visited by the random walk $\Tilde{X}^{n^r}$ before $m$. Let $t>0$, $d\in(0,\mathfrak{b})$ and introduce 
$\mathfrak{X}_{n}:=\sum_{|x|\leq(\log n)^3}\mathcal{L}_x^{\tau^{n^{\mathfrak{b}}}}$. Also define
\begin{align*}
    \mathfrak{B}^{n^{\mathfrak{b}}}_{\ell_n}(m):=\sup\{p\leq n^{\mathfrak{b}};\; B_{\ell_n}^p\leq m\}.
\end{align*}
One can see that on the event $\{\tau^{n^d}\leq\lfloor t\kappa_n\rfloor-\mathfrak{X}_n,\; \tau^{n^{\mathfrak{b}}}>\lfloor t\kappa_n\rfloor\}$, we have $N^{\lfloor t\kappa_n\rfloor}\leq \mathfrak{B}^{n^{\mathfrak{b}}}_{\ell_n}(\Tilde{R}^{n^{\mathfrak{b}}}_{\lfloor t\kappa_n\rfloor}-1)$ and since $\sum_{x\in C^{n^{\mathfrak{b}}}_{\ell_n}}\mathcal{L}_x^{\tau^{n^{\mathfrak{b}}}}\leq\mathfrak{X}_{n}$ (assuming that we are on the event $\{(\mathcal{B}^p_{\ell_n})_{p\leq n^{\mathfrak{b}}}\textrm{ is non-decreasing,}$ $\mathcal{B}^{n^d}_{\ell_n}\not=\varnothing,\; \sup_{x\in C^{n^{d}}_{\ell_n}}|x|\leq(\log n)^3\}$, whose quenched probability goes to $1$ as $n\to\infty$ thanks to Lemma \ref{NDSet} and Lemma \ref{PropSetRoots1}), we have $N^{\lfloor t\kappa_n\rfloor}\geq  \mathfrak{B}^{n^{\mathfrak{b}}}_{\ell_n}(\Tilde{R}^{n^{\mathfrak{b}}}_{\lfloor t\kappa_n\rfloor-\mathfrak{X}_{n}}-1)$. Note that we can choose $d$ and $\mathfrak{b}$ such that $0<d<1<\mathfrak{b}<\kappa\land 2$ and satisfying $\lim_{n\to\infty}\P^{\mathcal{E}}(\forall t\in[\delta',M'],\; \tau^{n^d}\leq\lfloor t\kappa_n\rfloor-\mathfrak{X}_n,\; \tau^{n^{\mathfrak{b}}}>\lfloor t\kappa_n\rfloor)=1$ (see the proof of \eqref{ProbaTA1}) so for any $\varepsilon_1>0$
\begin{align}\label{ApproxLocalTime}
    \lim_{n\to\infty}\P^{\mathcal{E}}\Big(\forall t\in[\delta',M'],\;\frac{1}{n} \mathfrak{B}^{n^{\mathfrak{b}}}_{\ell_n}(\Tilde{R}^{n^{\mathfrak{b}}}_{\lfloor t\kappa_n\rfloor-\mathfrak{X}_{n}}-1)-\varepsilon_1\leq\frac{1}{n}N^{\lfloor t\kappa_n\rfloor}\leq\frac{1}{n}\mathfrak{B}^{n^{\mathfrak{b}}}_{\ell_n}(\Tilde{R}^{n^{\mathfrak{b}}}_{\lfloor t\kappa_n\rfloor}-1)+\varepsilon_1\Big)=1.
\end{align}
Since $\mathfrak{b}<\kappa\land 2$, we have $\lim_{n\to\infty}\E^{\mathcal{E}}[\mathfrak{X}_n/\kappa_n]=0$. Hence, $\mathfrak{X}_n$ does not have any influence so we will no longer refer to it. Now, similarly to the proof of the convergence of $((T^{\lfloor\alpha n\rfloor}/\kappa_n;\; \alpha\geq 0))_{n\in\N^*}$, we have, under $\P$, that $(\Tilde{R}^{n^{\mathfrak{b}}}_{\lfloor t\kappa_n\rfloor}-1;\; t\in[\delta',M'])$ is distributed as $(\Bar{F}^{n^{\mathfrak{b}}}(\lfloor t\kappa_n\rfloor);\; t\in[\delta',M'])$ where $\Bar{F}^{n^{\mathfrak{b}}}(\lfloor t\kappa_n\rfloor):=B^{n^{\mathfrak{b}}}_{\ell_n}\land\Bar{F}(\lfloor t\kappa_n\rfloor)$, we recall that $\Bar{F}(\lfloor t\kappa_n\rfloor)=\sup\{p\in\N;\; F_p\leq\lfloor t\kappa_n\rfloor\}$ with $F_q=\sum_{i=1}^{q}\sum_{x\in\mathcal{R}_{1,i}}(\beta_i(x)+\sum_{y;\; y^*=x}\beta_i(y))$, $F_0=0$ and also recall that $\mathcal{F}^1=((\mathcal{R}_{1,i},(\beta_i(x);\; x\in\mathcal{R}_{1,i})))_{i\in\N^*}$, which satisfies hypothesis \ref{H1}, \ref{H2} and \ref{H3}, see Fact \ref{GWF_Hyp}. \\
Finally, note that $(\mathfrak{B}^{n^{\mathfrak{b}}}_{\ell_n}(m))_{m\in\N^*}$ is non-decreasing so thanks to Lemma \ref{RootNumber}, together with the fact that $\mathfrak{b}>1$, we obtain that $((\mathfrak{B}^{n^{\mathfrak{b}}}_{\ell_n}(\lfloor tn\rfloor)/n;\; t\geq 0))_{m\in\N^*}$ converges in law to $(t/W_{\infty};\; t\geq 0)$. Therefore, by \eqref{ApproxLocalTime} and thanks to Proposition \ref{PropArbresMulti}, together with the Skorokhod's representation theorem, we obtain the result.

\vspace{0.3cm}

\noindent Note that the joint convergence in law, under $\P^*$, of $((T^{\lfloor\alpha n\rfloor}/\kappa_n;\; \alpha\geq 0),(\mathcal{L}^{\lfloor t\kappa_n\rfloor}/n;\; t\geq 0))_{n\in\N^*}$ follows immediately. 

\vspace{0.3cm}

\noindent\textbf{Convergence of $((T^{\lfloor\alpha n\rfloor}/\kappa_n;\; \alpha\geq 0),(\mathcal{L}^{\lfloor t\kappa_n\rfloor}/n;\; t\geq 0))_{n\in\N^*}$ under $\P^{\mathcal{E}}$}

\vspace{0.1cm}

\noindent The proof is similar to the ones of Theorem 1.1 and Theorem 6.1 in \cite{AidRap} and of Theorem 1 in \cite{deRaph1}. First, note that thanks to \eqref{ProbaTA1} and \eqref{ApproxLocalTime}, together with the fact that $\sum_{i=1}^{B^{\lfloor\alpha n\rfloor}_{\ell_n}}\sum_{x\in\mathcal{R}^{\lfloor\alpha n\rfloor}_i}\mathcal{L}_{i,x}^{\lfloor\alpha n\rfloor}=\sum_{i=1}^{B^{\lfloor\alpha n\rfloor}_{\ell_n}}\sum_{x\in\Tilde{\mathcal{R}}^{n^{\mathfrak{b}}}_i}\Tilde{\mathcal{L}}_{i,x}^{n^{\mathfrak{b}}}$ (which is also equal to the time spent by the random walk $\Tilde{X}^{n^{\mathfrak{b}}}$ in $\cup_{i=1}^{B^{\lfloor\alpha n\rfloor}_{\ell_n}}\T^p_i$), it is enough to show the quenched convergence of the sequence of processes 
$\mathbb{B}=((\mathbb{B}_n)_{n\in\N^*}$ where $\mathbb{B}_n=(\frac{1}{\kappa_n}\sum_{i=1}^{B^{\lfloor\alpha n\rfloor}_{\ell_n}}\sum_{x\in\Tilde{\mathcal{R}}^{n^{\mathfrak{b}}}_i}\Tilde{\mathcal{L}}_{i,x}^{n^{\mathfrak{b}}};\; \alpha\in[\delta,M]),(\frac{1}{n}\mathfrak{B}^{n^{\mathfrak{b}}}_{\ell_n}(\Tilde{R}^{n^{\mathfrak{b}}}_{\lfloor t\kappa_n\rfloor}-1);\; t\in[\delta',M']))$. For that, take a continuous and bounded function $F:D([\delta,M])\times D([\delta',M'])\to\R$ and let us show that $\Pb$-almost surely, $(\E^{\mathcal{E}}[F(\mathbb{B}_n)]-\E[F(\mathbb{B}_n)])$ converges towards $0$ as $n\to\infty$. This will be enough since we have just proved the convergence of $(\E[F(\mathbb{B}_n)])_n$. \\ 
First, note that we can restrict ourselves to the case $B^{n^{\mathfrak{b}}}_{\ell_n}\leq n^2$ and to the vertices $z$ in generation $\ell_n$ such that $V(z)\geq\mathfrak{z}\ell_n$ for some constant $\mathfrak{z}>0$ only depending on the function $\psi$ (recall that we obtain in Lemma \ref{RootNumber} a more precise annealed approximation of $B^{n^{\mathfrak{b}}}_{\ell_n}$ but here, we need one under the quenched probability and it does not require such precision). Indeed, thanks to the Markov inequality
\begin{align*}
    \P^{\mathcal{E}}\big(B^{n^{\mathfrak{b}}}_{\ell_n}>n^2\big)\leq\frac{1}{n^2}\E^{\mathcal{E}}\big[B^{n^{\mathfrak{b}}}_{\ell_n}\un_{\{\sup_{x\in\mathcal{B}^{n^{\mathfrak{b}}}_{(\log n)^2}}|x|\leq(\log n)^3\}}\big]+\P^{\mathcal{E}}\big(\sup_{x\in\mathcal{B}^{n^{\mathfrak{b}}}_{(\log n)^2}}|x|>(\log n)^3\big),
\end{align*}
and, by definition, $B^{n^{\mathfrak{b}}}_{\ell_n}\un_{\{\sup_{x\in\mathcal{B}^{n^{\mathfrak{b}}}_{(\log n)^2}}|x|\leq(\log n)^3\}}\leq\sum_{|x|\leq(\log n)^3}N_x^{T^{n^{\mathfrak{b}}}}$ so
\begin{align*}
    \P^{\mathcal{E}}\big(B^{n^{\mathfrak{b}}}_{\ell_n}>n^2\big)\leq\frac{1}{n^{2-\mathfrak{b}}}\sum_{k\leq(\log n)^3}W_k+\P^{\mathcal{E}}\big(\sup_{x\in\mathcal{B}^{n^{\mathfrak{b}}}_{(\log n)^2}}|x|>(\log n)^3\big),
\end{align*}
which, recalling that $\mathfrak{b}<\kappa\land 2$, goes to $0$, $\Pb$-almost surely (see Lemma \ref{PropSetRoots1} \ref{Fact2.2}). For the restriction to the vertices in generation $\ell_n$ with high potential, let $\mathfrak{z}:=|\psi(\mathfrak{h})|/2\mathfrak{h}$ (where $\mathfrak{h}>0$ is chosen such that $\psi(\mathfrak{h})<0$, which is possible since $\psi'(1)<0$) and note that
\begin{align*}
    \P\big(\exists\; |z|=\ell_n;\; V(z)<\mathfrak{z}\ell_n\big)\leq\Eb\Big[\sum_{|z|=\ell_n}\un_{\{V(z)<\mathfrak{z}\ell_n\}}\Big]\leq\Eb\Big[\sum_{|z|=\ell_n}e^{\mathfrak{h}\mathfrak{z}\ell_n-\mathfrak{h}|\psi(\mathfrak{h})|}\Big]=e^{-\ell_n|\psi(\mathfrak{h})|/2},
\end{align*}
and recalling $\ell_n=(\log n)^2$, we obtain that $\sum_{n\in\N}\P(\exists\; |z|=\ell_n;\; V(z)<\mathfrak{z}\ell_n)<\infty$, thus giving the convergence towards $0$ of $\P^{\mathcal{E}}(\exists\; |z|=\ell_n;\; V(z)<\mathfrak{z}\ell_n)$, $\Pb$-almost surely. \\
Considering what we have just showed, we are left to prove that $(\E^{\mathcal{E}}[F(\mathbb{B}_n)\un_{\{B^{n^{\mathfrak{b}}}_{\ell_n}\leq n^2,\min_{|z|=\ell_n}V(z)\geq\mathfrak{z}\ell_n\}}]-\E[F(\mathbb{B}_n)])$ goes to $0$, $\Pb$-almost surely, as $n\to\infty$. For that, define $\mathcal{A}_n$ to be the subset of $(\bigcup_{k>\ell_n}\N^k)^2$ such that $(U_1,U_2)\in\mathcal{A}_n$ if and only if for any $(x,y)\in(U_1,U_2)$, neither $x$ is an ancestor of $y$, nor $y$ is an ancestor of $x$ and introduce $\mathcal{C}_n:=\{(U_1,U_2)\in(\bigcup_{k>\ell_n}\N^k)^2
\setminus\mathcal{A}_n;\;\forall\; i\in\{1,2\},\;\#U_i\leq n^2,\;\forall\; x\in U_i:\; V(x_{\ell_n})\geq\mathfrak{z}\ell_n\}$. Note that $\Eb[\E^{\mathcal{E}}[F(\mathbb{B}_n)\un_{\{B^{n^{\mathfrak{b}}}_{\ell_n}\leq n^2,\min_{|z|=\ell_n}V(z)\geq\mathfrak{z}\ell_n\}}]^2]$ is smaller than
\begin{align*}
    \Eb\Big[\Big(\sum_{(U_1,U_2)\in\mathcal{A}_n}+\sum_{(U_1,U_2)\in\mathcal{C}_n}\Big)\E^{\mathcal{E}}[F(\mathbb{B}_n)\un_{\{\mathcal{B}^{n^{\mathfrak{b}}}_{\ell_n}=U_1\}}]\times\E^{\mathcal{E}}[F(\mathbb{B}_n)\un_{\{\mathcal{B}^{n^{\mathfrak{b}}}_{\ell_n}=U_2\}}]\Big].
\end{align*}
One can see that 
\begin{align*}
    &\sum_{(U_1,U_2)\in\mathcal{A}_n}\Eb\Big[\E^{\mathcal{E}}\big[F(\mathbb{B}_n)\un_{\{\mathcal{B}^{n^{\mathfrak{b}}}_{\ell_n}=U_1\}}\big]\times\E^{\mathcal{E}}\big[F(\mathbb{B}_n)\un_{\{\mathcal{B}^{n^{\mathfrak{b}}}_{\ell_n}=U_2\}}\big]\Big] \\ & =\sum_{(U_1,U_2)\in\mathcal{A}_n}\Eb\Big[\E^{\mathcal{E}}\big[F(\mathbb{B}_n)\big|\mathcal{B}^{n^{\mathfrak{b}}}_{\ell_n}=U_1\big]\E^{\mathcal{E}}\big[F(\mathbb{B}_n)\big|\mathcal{B}^{n^{\mathfrak{b}}}_{\ell_n}=U_2\big]\P^{\mathcal{E}}\big(\mathcal{B}^{n^{\mathfrak{b}}}_{\ell_n}=U_1\big)\P^{\mathcal{E}}\big(\mathcal{B}^{n^{\mathfrak{b}}}_{\ell_n}=U_2\big)\Big] \\ & =\Eb\Big[\P^{\mathcal{E}}\big(\mathcal{B}^{n^{\mathfrak{b}}}_{\ell_n}\in\mathcal{A}_n\big)^2\Big]\E[F(\mathbb{B}_n)]^2\leq\E[F(\mathbb{B}_n)]^2,
\end{align*}
where the second equality comes from that, for any $(U_1,U_2)\in\mathcal{A}_n$, by definition of the random walk $\Tilde{X}^{n^{\mathfrak{b}}}$, the random variables $\E^{\mathcal{E}}[F(\mathbb{B}_n)\big|\mathcal{B}^{n^{\mathfrak{b}}}_{\ell_n}=U_1]$ and $\E^{\mathcal{E}}[F(\mathbb{B}_n)\big|\mathcal{B}^{n^{\mathfrak{b}}}_{\ell_n}=U_2]$ are independent under $\Pb$ with common expectation $\E[F(\mathbb{B}_n)]$ and are also independent of $\P^{\mathcal{E}}(\mathcal{B}^{n^{\mathfrak{b}}}_{\ell_n}=U_1)\P^{\mathcal{E}}(\mathcal{B}^{n^{\mathfrak{b}}}_{\ell_n}=U_2)$. We also have
\begin{align*}
    &\Eb\Big[\sum_{(U_1,U_2)\in\mathcal{C}_n}\E^{\mathcal{E}}\big[F(\mathbb{B}_n)\un_{\{\mathcal{B}^{n^{\mathfrak{b}}}_{\ell_n}=U_1\}}\big]\times\E^{\mathcal{E}}\big[F(\mathbb{B}_n)\un_{\{\mathcal{B}^{n^{\mathfrak{b}}}_{\ell_n}=U_2\}}\big]\Big]\leq\|F\|_{\infty}^2 \\ & \times\sum_{\underset{\# U_1\leq n^2}{U_1\subset\cup_{k>\ell_n}\N^k}}\Eb\Big[\P^{\mathcal{E}}\big(\mathcal{B}^{n^{\mathfrak{b}}}_{\ell_n}=U_1\big)\sum_{U_2\subset\cup_{k>\ell_n}\N^k}\un_{\{\min_{x\in U_2} V(x_{\ell_n})\geq\mathfrak{z}\ell_n,\;\exists\;(u_1,u_2)\in U_1\times U_2:\; u_2\leq u_1\}}\P^{\mathcal{E}}\big(\mathcal{B}^{n^{\mathfrak{b}}}_{\ell_n}=U_2\big)\Big].
\end{align*}
Note that
\begin{align*}
    &\sum_{U_2\subset\cup_{k>\ell_n}\N^k}\un_{\{\min_{x\in U_2} V(x_{\ell_n})\geq\mathfrak{z}\ell_n,\;\exists\;(u_1,u_2)\in U_1\times U_2:\;u_2\leq u_1\}}\P^{\mathcal{E}}\big(\mathcal{B}^{n^{\mathfrak{b}}}_{\ell_n}=U_2\big) \\ & \leq\sum_{x\in U_1}\un_{\{V(x_{\ell_n})\geq\mathfrak{z}\ell_n\}}\sum_{U_2\subset\cup_{k>\ell_n}\N^k}\P^{\mathcal{E}}\big(\mathcal{B}^{n^{\mathfrak{b}}}_{\ell_n}=U_2,\; N_{x_{\ell_n}}^{\tau^{n^{\mathfrak{b}}}}\geq 1\big) \\ & \leq\sum_{x\in U_1}\un_{\{V(x_{\ell_n})\}\geq\mathfrak{z}\ell_n\}}\P^{\mathcal{E}}\big(N_{x_{\ell_n}}^{\tau^{n^{\mathfrak{b}}}}\geq 1\big). 
\end{align*}
Moreover, by Lemma \ref{LemmaEdge}, $\sum_{x\in U_1}\un_{\{V(x_{\ell_n})\}\geq\mathfrak{z}\ell_n\}}\P^{\mathcal{E}}(N_{x_{\ell_n}}^{\tau^{n^{\mathfrak{b}}}}\geq 1)\leq\sum_{x\in U_1}\un_{\{V(x_{\ell_n})\}\geq\mathfrak{z}\ell_n\}}e^{-V(x_{\ell_n})}$ which is smaller than $\#U_1e^{-\mathfrak{z}\ell_n}$.
Hence
\begin{align*}
    \Eb\Big[\sum_{(U_1,U_2)\in\mathcal{C}_n}\E^{\mathcal{E}}\big[F(\mathbb{B}_n)\un_{\{\mathcal{B}^{n^{\mathfrak{b}}}_{\ell_n}=U_1\}}\big]\times\E^{\mathcal{E}}\big[F(\mathbb{B}_n)\un_{\{\mathcal{B}^{n^{\mathfrak{b}}}_{\ell_n}=U_2\}}\big]\Big]&\leq\|F\|_{\infty}^2\sum_{\underset{\# U_1\leq n^2}{U_1\subset\cup_{k>\ell_n}\N^k}}\#U_1e^{-\mathfrak{z}\ell_n}\P\big(\mathcal{B}^{n^{\mathfrak{b}}}_{\ell_n}=U_1\big) \\ & \leq\|F\|_{\infty}^2 n^2e^{-\mathfrak{z}\ell_n},
\end{align*}
thus giving $\Eb[\sum_{(U_1,U_2)\in\mathcal{C}_n}\E^{\mathcal{E}}[F(\mathbb{B}_n)\un_{\{\mathcal{B}^{n^{\mathfrak{b}}}_{\ell_n}=U_1\}}]\times\E^{\mathcal{E}}[F(\mathbb{B}_n)\un_{\{\mathcal{B}^{n^{\mathfrak{b}}}_{\ell_n}=U_2\}}]]\leq n^2e^{-\mathfrak{z}\ell_n}$. It follows that the sum $\sum_{n\in\N}\Pb((\E^{\mathcal{E}}[F(\mathbb{B}_n)\un_{\{B^{n^{\mathfrak{b}}}_{\ell_n}\leq n^2,\min_{|z|=\ell_n}V(z)\geq\mathfrak{z}\ell_n\}}]-\E[F(\mathbb{B}_n)])>\varepsilon_1)$ is finite for any $\varepsilon_1>0$ and the proof is completed.
\end{proof}

\noindent We now turn to the proof of Theorem \ref{Th3}.

\begin{proof}[Proof of Theorem \ref{Th3}]
Without loss of generality, let us prove that for any $\varepsilon>0$
\begin{align*}
    \P^*\Big(\sup_{1\leq p\leq n}\Big|R_{\tau^p}-\frac{\bm{c}_{\infty}}{2}\tau^p\Big|>\varepsilon\kappa_n\Big)\underset{n\to\infty}{\longrightarrow}0.
\end{align*}
First, note that $R_{\tau^p}\leq\tau^p$ so for any $\delta\in(0,1)$
\begin{align*}
    \P^*\Big(\sup_{1\leq p\leq\lfloor\delta n\rfloor}\Big|R_{\tau^p}-\frac{\bm{c}_{\infty}}{2}\tau^p\Big|>\varepsilon\kappa_n\Big)\leq\P^*\big(\tau^{\lfloor\delta n\rfloor}>\varepsilon\kappa_n/2\big)+\P^*\big(\tau^{\lfloor\delta n\rfloor}>\varepsilon\kappa_n/\bm{c}_{\infty}\big),
\end{align*}
which gives $\lim_{\delta\to 0}\limsup_{n\to\infty}\P^*(\sup_{1\leq p\leq\lfloor\delta n\rfloor}|R_{\tau^p}-\frac{\bm{c}_{\infty}}{2}\tau^p|>\varepsilon\kappa_n)=0$. Now, similarly to \eqref{ProbaTA1}, we have
\begin{align*}
    \sup_{p\in\{\lfloor\delta n\rfloor,\ldots,n\}}\Big|R_{\tau^p}-\sum_{i=1}^{B_{\ell_n}^p}\sum_{x\in\Tilde{\mathcal{R}} ^{p}_i}1\Big|\leq\sum_{x\in C^{\lfloor\delta n\rfloor}_{\ell_n}}N_x^{\tau^{\lfloor\delta n\rfloor}},
\end{align*}
so
\begin{align*}
    \lim_{n\to\infty}\P^{\mathcal{E}}\Big(\sup_{p\in\{\lfloor\delta n\rfloor,\ldots,n\}}\Big|R_{\tau^p}-\sum_{i=1}^{B_{\ell_n}^p}\sum_{x\in\Tilde{\mathcal{R}} ^{p}_i}1\Big|>\varepsilon\kappa_n\Big)=0,
\end{align*}
and thanks to \eqref{ProbaTA1}
\begin{align*}
     &\limsup_{n\to\infty}\P^*\Big(\sup_{p\in\{\lfloor\delta n\rfloor,\ldots,n\}}\Big|R_{\tau^p}-\frac{\bm{c}_{\infty}}{2}\tau^p\Big|>\varepsilon\kappa_n\Big) \\ & \leq\limsup_{n\to\infty}\P^*\Big(\sup_{p\in\{\lfloor\delta n\rfloor,\ldots,n\}}\Big|\sum_{i=1}^{B_{\ell_n}^p}\sum_{x\in\Tilde{\mathcal{R}}^{n}_i}1-\frac{\bm{c}_{\infty}}{2}\sum_{i=1}^{B_{\ell_n}^p}\sum_{x\in\Tilde{\mathcal{R}} ^{n}_i}\Tilde{\mathcal{L}}^{n}_{i,x}\Big|>\varepsilon\kappa_n\Big).
\end{align*}
Moreover, splitting the probability according to the value of $B^p_{\ell_n}$, we have, for any $\delta',M>0$
\begin{align*}
    &\P^*\Big(\sup_{p\in\{\lfloor\delta n\rfloor,\ldots,n\}}\Big|\sum_{i=1}^{B_{\ell_n}^p}\sum_{x\in\Tilde{\mathcal{R}}^{n}_i}1-\frac{\bm{c}_{\infty}}{2}\sum_{i=1}^{B_{\ell_n}^p}\sum_{x\in\Tilde{\mathcal{R}} ^{n}_i}\Tilde{\mathcal{L}}^{n}_{i,x}\Big|>\varepsilon\kappa_n\Big) \\ & \leq\P^*\Big(\sup_{q\in\{\lfloor\delta' n\rfloor,\ldots,\lfloor nM\rfloor\}}\Big|\sum_{i=1}^{q}\sum_{x\in\Tilde{\mathcal{R}}^{n}_i}1-\frac{\bm{c}_{\infty}}{2}\sum_{i=1}^{q}\sum_{x\in\Tilde{\mathcal{R}} ^{n}_i}\Tilde{\mathcal{L}}^{n}_{i,x}\Big|>\varepsilon\kappa_n\Big) \\ & +\P^*\big(\exists\; p\in\{\lfloor\delta n\rfloor,\ldots,n\}:\; B^p_{\ell_n}\not\in\{\lfloor\delta' n\rfloor,\ldots,\lfloor Mn\rfloor\}\big).
\end{align*}
Finally, note that $((\sum_{i=1}^{q}\sum_{x\in\Tilde{\mathcal{R}} ^{n}_i}1)_{q\in\{\lfloor\delta' n\rfloor,\ldots,\lfloor Mn\rfloor\}},(\sum_{i=1}^{q}\sum_{x\in\Tilde{\mathcal{R}}^{n}_i}\Tilde{\mathcal{L}}^{n}_{i,x})_{q\in\{\lfloor\delta' n\rfloor,\ldots,\lfloor Mn\rfloor\}})$ is, under $\P$, distributed as $((\sum_{i=1}^{q}\sum_{x\in\mathcal{R}_{1,i}}1)_{p\in\{\lfloor\delta n\rfloor,\ldots,n\}},(F_{q})_{p\in\{\lfloor\delta n\rfloor,\ldots,n\}})$ where we recall that for any $q\in\N$, $F_q=\sum_{i=1}^{q}\sum_{x\in\mathcal{R}_{1,i}}(\beta_i(x)+\sum_{y;\; y^*=x}\beta_i(y))$ and $\mathcal{F}^1=((\mathcal{R}_{1,i},(\beta_i(x);\; x\in\mathcal{R}_{1,i})))_{i\in\N^*}$ satisfies hypothesis \ref{H1}, \ref{H2} and \ref{H3}, see Fact \ref{GWF_Hyp} 
Moreover, $\E[\sum_{x\in\T}\un_{\{\min_{e<z<x}N_z^{\tau^1}\geq 2\}}]=\bm{c}_{\infty}/C_{\infty}$ thanks to \cite{AidRap} and \cite{deRaph1}. Hence, Proposition \ref{PropArbresMulti}, together with Lemma \ref{RootNumber} by taking $\delta'$ small enough and $M$ large enough allow to complete the proof.
\end{proof}

\vspace{0.2cm}

\noindent Finally, we end this section with the proof of Corollary \ref{Coro1}, providing an equivalent of the probability $\P^{\mathcal{E}}(X_{2n+1}=e^*)$ when $n\to\infty$.

\begin{proof}[Proof of Corollary \ref{Coro1}]
The proof follows a similar approach as the one of Corollary 1.4 in \cite{Hu2017Corrected}. We first prove that $(\mathcal{L}^n/\gamma_n)_{n\in\N^*}$ is uniformly integrable under $\P^{\mathcal{E}}$ (recalling that $\gamma_n=n^{1/(\kappa\land 2)}$ if $\kappa\not=2$ and $\gamma_n=\sqrtBis{(n\log n)}$ if $\kappa=2$), so that Theorem \ref{Th2} provides the convergence of $(\E^{\mathcal{E}}[\mathcal{L}^n/\gamma_n])_{n\in\N^*}$. For that, it is enough to show that $\sup_{n\in\N^*}\E^{\mathcal{E}}[(\mathcal{L}^n/\gamma_n)^d]<\infty$, with $d\in(1,\infty)$. Note that for any $n\in\N^*$ and $\varepsilon\in(0,1)$, we have 
\begin{align*}
    \E^{\mathcal{E}}\big[(\mathcal{L}^n)^d\big]\leq(\gamma_n/\varepsilon)^d+\sum_{k\geq\lfloor\gamma_n/\varepsilon\rfloor}dk^{d-1}\P^{\mathcal{E}}\big(\mathcal{L}^n\geq k\big)=(\gamma_n/\varepsilon)^d+\sum_{k\geq\lfloor\gamma_n/\varepsilon\rfloor}dk^{d-1}\P^{\mathcal{E}}\big(T^k\leq n\big),
\end{align*}
and thanks to the strong Markov property, $\P^{\mathcal{E}}(T^k\leq n)\leq\P^{\mathcal{E}}(T^{\lfloor\gamma_n/\varepsilon\rfloor}\leq n)^{\lfloor k/\lfloor\gamma_n/\varepsilon\rfloor\rfloor}$, for any integer $k\geq\lfloor\gamma_n/\varepsilon\rfloor$. Thanks to Theorem \ref{Th1}, we can find $\varepsilon_0,c_0\in(0,1)$ and $n_0\in\N^*$ such that for all $n\geq n_0$, $\P^{\mathcal{E}}(T^{\lfloor\gamma_n/\varepsilon_0\rfloor}\leq n)\leq 1-c_0$ and this yields, for any $n\geq n_0$
\begin{align*}
    \E^{\mathcal{E}}\big[(\mathcal{L}^n)^d\big]\leq C_{d,\varepsilon_0}(\gamma_n)^d\sum_{k\geq 0}dk^{d-1}(1-c_0)^k\leq C'_{d,\varepsilon_0}(\gamma_n)^d,
\end{align*}
for some positive constants $C_{d,\varepsilon_0}$ and $C'_{d,\varepsilon_0}$. Then, by Theorem \ref{Th2} and \eqref{LaplaceMittag-Leffler} \\
\begin{align*}
    \E^{\mathcal{E}}\big[\mathcal{L}^{2n+1}\big]\sim_{n\to\infty}\left \{
   \begin{array}{l c l}
      \sqrtBis{n}\frac{1}{\Gamma(1+1/2)}\frac{\sqrtBis{(2c_0)}}{W_{\infty}}  & \text{if} & \kappa>2 \\ \\
       \sqrtBis{(n\log n)}\frac{1}{\Gamma(1+1/2)}\frac{1}{W_{\infty}}\sqrtBis{(C_{\infty}c_2)} & \text{if} & \kappa=2 \\ \\
       n^{1/\kappa}\frac{1}{\Gamma(1+1/\kappa)}\frac{1}{W_{\infty}}(C_{\infty}c_{\kappa}|\Gamma(1-\kappa)|)^{1/\kappa}  & \text{if} & \kappa\in(1,2).
   \end{array}
   \right .
\end{align*}
Then, note that $\E^{\mathcal{E}}[\mathcal{L}^{2n+1}]=\sum_{k=0}^{n}\P^{\mathcal{E}}(X_{2k+1}=e^*)$ and for all $s\in[0,1)$, $\sum_{k\geq 0}s^k\P^{\mathcal{E}}(X_{2k+1}=e^*)<\infty$. Besides, one can notice that $\P^{\mathcal{E}}(X_{2k+1}=e^*)=\P_{e^*}^{\mathcal{E}}(X_{2k+2}=e^*)$ so the function $k\in\N\mapsto\P^{\mathcal{E}}(X_{2k+1}=e^*)$ is non-decreasing. Hence, a Tauberian theorem (see \cite{Feller2}, Theorem 5) leads to the result and the proof is completed.
\end{proof}

\section{Proof of Proposition \ref{PropArbresMulti}}\label{ProofProp}

This section is dedicated to the proof of Proposition \ref{PropArbresMulti}. 
Recall that  $\mathcal{F}=((\mathcal{T}_i,(\beta(x);x\in\mathcal{T}_i));i\in\N^*)$ is a multi-type Galton-Watson forest (with infinitely many types $(\beta(x);\; x\in\mathcal{T}_i)$, $\beta(x)\in\N^*$ and $\mathcal{T}_i$ is rooted at $\VectCoord{e}{i}$) such that 
\begin{enumerate}[start=1,label={(\bfseries H\arabic*)}]
    \item $\Pbis$-almost surely, $\forall\; i\in\N^*$, $\beta(\VectCoord{e}{i})=1$;
    \item $\Ebis[\sum_{x\in\mathcal{T}_1}\un_{\{G^1_{1}(x)=1,\; \beta(x)=1\}}]=1$ where $G^1_i(x)$ denotes the number of vertices of type 1 in genealogical line of $x$:
    \begin{align*}
        \forall\; u\in\mathcal{T}_i\setminus\{\VectCoord{e}{i}\},\;G^1_i(u)=\sum_{\VectCoord{e}{i}\leq z<u}\un_{\{\beta(z)=1\}}\;\;\textrm{ and }\;\; G_i^1(\VectCoord{e}{i})=0;
    \end{align*}
    \item $\Ebis[\sum_{x\in\mathcal{T}_1}\beta(x)\un_{\{G^1_1(x)=1\}}]=\nu<\infty$.
\end{enumerate}
\noindent Also recall that for any $p\in\N^*$, $\mathcal{F}_p:=(\mathcal{T}_i)_{i\in\{1,\ldots,p\}}$ is the family composed of the first $p$ trees of the forest $\mathcal{F}$ and for any $x\in\mathcal{T}_i$, $\beta^{\star}(x):=\beta(x)+\sum_{y\in\mathcal{T}_i;\; y^*=x}\beta(y)$. Also recall that $F_p=\sum_{x\in\mathcal{F}_p}\beta^{\star}(x)$ and $\Bar{F}(m):=\sup\{p\in\N^*;\; F_{p}\leq m\}$.

\vspace{0.3cm}

\noindent Before proving Proposition \ref{PropArbresMulti}, let us introduce some definitions and notations. Following the idea of L. de Raphélis \cite{deRaphMulti} and E. Aïdékon and L. de Raphélis \cite{AidRap}, we construct, starting from the tree $(\mathcal{T}_i,(\beta(x);\; x\in\mathcal{T}_i))$ with $i\in\N^*$, a new tree $(\mathfrak{T}_i,(\Type{x};\; x\in\mathfrak{T}_i))$, with $\Type{x}\in\{0,1\}$ and composed of $\sum_{x\in\mathcal{T}_i}\beta^{\star}(x)$ vertices. It needs two steps. First, we build the skeleton $(\Tilde{\mathcal{T}}_i,(\bm{\beta}(x);\; x\in\Tilde{\mathcal{T}}_i))$, with $\bm{\beta}(x)\in\{0,1\}\times\N^*$, of $(\mathfrak{T}_i,(\Type{x};\; x\in\mathfrak{T}_i))$ as follows
\begin{itemize}
    \item $\Tilde{\mathcal{T}}_i$ is a tree rooted at $\VectCoord{e}{i}$;
    \item as $\mathcal{T}_i$, the tree $\Tilde{\mathcal{T}}_i$ is composed of $\#\mathcal{T}_i$ vertices. For any $j\in\{0,\ldots,\#\mathcal{T}_i-1\}$, if $u(j)$ denotes the $(j+1)$-th vertex of $\mathcal{T}_i$ and $v(j)$ the $(j+1)$-th vertex of $\Tilde{\mathcal{T}}_i$ (for the depth-first search order), then the generation of $v(j)$ is given by $G^1_i(u(j))$;
    \item for any $j\in\{0,\ldots,\#\mathcal{T}_i-1\}$, if $\beta(u(j))\not=1$, then $v(j)$ is a leaf of the tree $\Tilde{\mathcal{T}}_i$, meaning it has no descendant. We then assign a new type to $v(j)$: $\bm{\beta}(v(j)):=(0,\beta^{\star}(u(j)))$. Otherwise, $\beta(u(j))=1$ and $v(j)$ inherits from the type of $u(j)$: $\bm{\beta}(v(j))=(1,\beta^{\star}(u(j)))$.
\end{itemize}
The skeleton $(\Tilde{\mathcal{T}}_i,(\bm{\beta}(x);\; x\in\Tilde{\mathcal{T}}_i))$ has been built, we can now move to $(\mathfrak{T}_i,(\Type{x};\; x\in\mathfrak{T}_i))$. First, for any $x\in\Tilde{\mathcal{T}}_i$, let $\bm{\beta}(x):=(\VectCoord{\beta}{1}(x),\VectCoord{\beta}{2}(x))\in\{0,1\}\times\N^*$. Then
\begin{itemize}
    \item $\mathfrak{T}_i$ is a tree rooted at $\VectCoord{e}{i}$ and $\Tilde{\mathcal{T}}_i$ is a sub-tree of $\mathfrak{T}_i$. For any $x\in\Tilde{\mathcal{T}}_i\cap\mathfrak{T}_i$, $\Type{x}:=\VectCoord{\beta}{1}(x)$.
    \item Let $x\in\Tilde{\mathcal{T}}_i\cap\mathfrak{T}_i$ such that $\mathfrak{t}(x)=1$. We add $\VectCoord{\beta}{2}(x)-1$ new vertices to the set of children of $x$ and the type $\Type{y}$ of each vertex $y$ among these new children is equal to $0$. Moreover, for any vertex $x\in\Tilde{\mathcal{T}}_i\cap\mathfrak{T}_i$ such that $\mathfrak{t}(x)=0$, we add $\beta^{2}(x)-1$ new vertices to the set of children of $x^*$ (the parent of $x$) and the type $\Type{y}$ of each vertex $y$ among these new children is equal to $0$. By definition, each new vertex we added is a leaf of $\mathfrak{T}_i$.
\end{itemize}
Note that the total number $\#\mathfrak{T}_i$ of vertices of the tree $\mathfrak{T}_i$ is equal to $\sum_{x\in\Tilde{\mathcal{T}}_i}\VectCoord{\beta}{2}(x)$, which is nothing but $\sum_{x\in\mathcal{T}_i}\beta^{\star}(x)$. Also, $\mathcal{F}$ is a multi-type Galton-Watson forest so does $\mathfrak{F}:=((\mathfrak{T}_i,(\Type{x};\; x\in\mathfrak{T}_i));\; i\in\N^*)$. \\
Introduce
\begin{align}
    \mathfrak{R}:=\sum_{x\in\mathfrak{T}_1}\un_{\{|x|=1\}}\;\;\textrm{ and }\;\;\mathfrak{R}^1:=\sum_{x\in\mathfrak{T}_1}\un_{\{|x|=1,\; \mathfrak{t}(x)=1\}},
\end{align}
respectively the reproduction law and the reproduction law restricted to vertices of type 1 associated to this forest.

\begin{lemm}\label{Fact1} We have:
\begin{enumerate}[label=(\roman*)] 
    \item\label{Fact1.1} $\Ebis[\mathfrak{R}]=2\nu$;
    \item\label{Fact1.2} $\Ebis[\mathfrak{R}^1]=1$.
\end{enumerate}
\end{lemm}

\begin{proof}
    \noindent For \ref{Fact1.1}, by definition
\begin{align*}
    \mathfrak{R}&=\beta^{\star}(\VectCoord{e}{1})-1+\sum_{x\in\mathcal{T}_1}\beta^{\star}(x)\un_{\{G^1_1(x)=1,\; \beta(x)\not=1\}}+\sum_{x\in\mathcal{T}_1}\un_{\{G^1_1(x)=1,\; \beta(x)=1\}}\\ & = \sum_{x\in\mathcal{T}_1;\; |x|=1}\beta(x)\un_{\{G^1_1(x)=1\}}+\sum_{x\in\mathcal{T}_1}\;\sum_{y\in\mathcal{T}_1;\; y^*=x}\beta(y)\un_{\{G^1_1(x)=1,\; \beta(x)\not=1\}}+\sum_{x\in\mathcal{T}_1
    }\beta(x)\un_{\{G^1_1(x)=1\}},
\end{align*}
where we have used, for the second equality, that $\beta(\VectCoord{e}{1})=1$ implying that for any $x\in\mathcal{T}_1$ such that $|x|=1$, we have necessarily $G^1_1(x)=1$. Now, it suffices to note that 
\begin{align*}
    \sum_{x\in\mathcal{T}_1}\;\sum_{y\in\mathcal{T}_1;\; y^*=x}\beta(y)\un_{\{G^1_1(x)=1,\; \beta(x)\not=1\}}&=\sum_{x\in\mathcal{T}_1}\;\sum_{y\in\mathcal{T}_1;\; y^*=x}\beta(y)\un_{\{G^1_1(y)=1\}}\\ & =\sum_{x\in\mathcal{T}_1;\; |x|>1}\beta(x)\un_{\{G^1_1(x)=1\}},
\end{align*}
to get $\mathfrak{R}=2\sum_{x\in\mathcal{T}_1}\beta(x)\un_{\{G^1_1(x)=1\}}$ which yields $\Ebis[\mathfrak{R}]=2\nu$. \\
For \ref{Fact1.2}, one can notice that none of the modifications we made to the initial forest $(\mathcal{T}_1,(\beta(x);\; x\in\mathcal{T}_1))$ has an incidence on the number of vertices of type $1$:
\begin{align*}
    \sum_{x\in\mathcal{T}_1}\un_{\{G^1_1(x)=1,\; \beta(x)=1\}}=\sum_{x\in\Tilde{\mathcal{T}_1}}\un_{\{|x|=1,\; \VectCoord{\beta}{1}(x)=1\}}=\sum_{x\in\mathfrak{T}_1}\un_{\{|x|=1,\; \Type{x}=1\}},
\end{align*}
thus giving the result since $\Ebis[\sum_{x\in\mathcal{T}_1}\un_{\{G^1_1(x)=1,\; \beta(x)=1\}}]=1$.
\end{proof}
\noindent We are now ready to prove our proposition.
\begin{proof}[Proof of Proposition \ref{PropArbresMulti}]
First, introduce the forest $\mathfrak{F}^1:=\{x\in\mathfrak{F};\; \Type{x}=1\}$ with reproduction law $\mathfrak{R}^1$ and for any $p\in\N^*$, denote by $\mathfrak{F}^1_p$ (resp. $\mathfrak{F}_p$) the finite forest composed of the $p$ first trees of $\mathfrak{F}^1$ (resp. $\mathfrak{F}$). Moreover, $\#\mathfrak{F}^1_p$ (resp. $\#\mathfrak{F}_p$) stands for its cardinal. One can note that, by construction, $F_p=\sum_{x\in\mathcal{F}_p}\beta^{\star}(x)=\#\mathfrak{F}_p$ (with the convention $F_0=0$). Moreover, we can easily link $\#\mathfrak{F}^1$ to $\#\mathfrak{F}$. Indeed, if, for any $j\in\N$, $\mathfrak{R}(j)$ denotes the number of children of the $(j+1)$-th vertex of the forest $\mathfrak{F}$, then we have
\begin{align}\label{ForestType}
    \#\mathfrak{F}_p=p+\mathfrak{D}(\#\mathfrak{F}^1_p)\;\;\textrm{ with }\;\;\mathfrak{D}(k):=\sum_{j=0}^{k-1}\mathfrak{R}(j)\;\;\forall\; k\in\N^*\;\;\textrm{ and }\;\; \mathfrak{D}(0):=0.
\end{align}
The reason why \eqref{ForestType} holds is that, by construction, if the $(j+1)$-th vertex $x$ of $\mathfrak{F}$ is satisfies $\mathfrak{t}(x)=0$, then $\mathfrak{R}(j)=0$. If $\mathfrak{R}^1(j)$ denotes the number of children of the $(j+1)$-th vertex of $\mathfrak{F}^1$, then $V^1(k):=\sum_{j=0}^{k-1}(\mathfrak{R}^1(j)-1)$ is the Lukasiewicz path associated to the forest $\mathfrak{F}^1$. One can notice that $\#\mathfrak{F}^1_p=\inf\{k\in\N^*;\; -V^1_k=p\}=\inf\{k\in\N^*;\; -V^1_k\geq p\}$. \\
Let us now recall that for any $m\in\N^*$, $\Bar{F}(m)=\sup\{p\in\N;\; F_{p}\leq m\}$. As we did for $F_p$, we want to link $\Bar{F}(m)$ with the Lukasiewicz path associated to $\mathfrak{F}^1$. For that, let us introduce
\begin{align}\label{InvDQ}
    \Bar{\mathfrak{D}}(m):=\sup\{k\in\N;\; \mathfrak{D}(k)\leq m\}.
\end{align}
By definition, we have, for any $\mathfrak{g}<m$
\begin{align}\label{EncadreSupForest}
    \mathfrak{g}\land\max_{1\leq k\leq\Bar{\mathfrak{D}}(m-\mathfrak{g})}\big(-V^1_k\big)\leq \Bar{F}(m)\leq\max_{1\leq k\leq\Bar{\mathfrak{D}}(m)}\big(-V^1_k\big).
\end{align}
Hence, to obtain Proposition \ref{PropArbresMulti}, we need to know the behaviour of $-V^1$. Thanks to a well known invariance principal (depending on the behavior of the reproduction law $\mathfrak{R}^1$, which is nothing but $\sum_{x\in\mathcal{T}_1}\un_{\{G^1_1(x)=1,\; \beta(x)=1\}}$), we have, in law under $\Pbis$, for the Skorokhod topology on $D([0,\infty))$
\begin{itemize}
    \item if $\Ebis[(\mathfrak{R}^1)^2]<\infty$, then the variance $\sigma_1^2$ of $\mathfrak{R}^1$ exists, belongs to $(0,\infty)$ and
        \begin{align*}
            \Big(-\frac{1}{\sqrtBis{n}}V^{1}(\lfloor nt\rfloor);\;t\geq 0\Big)\underset{n\to\infty}{\longrightarrow}\sqrtBis{(\sigma_1^2)}B,
        \end{align*}
        \item and if $\lim_{m\to\infty}m^{\gamma}\Pbis(\mathfrak{R}^1>m)=\mathfrak{c}_{\gamma}\in(0,\infty)$ exists for some $\gamma\in(1,2]$, then 
        \begin{align}\label{ConvLuka2}
            \Big(-\frac{1}{\gamma_n}V^{1}(\lfloor nt\rfloor);\; t\geq 0\Big)\underset{n\to\infty}{\longrightarrow}\Bar{\mathfrak{c}}_{\gamma}\VectCoord{\Tilde{Y}}{\gamma},
        \end{align}
\end{itemize}
where, if $\gamma=2$, then $\gamma_n=\sqrtBis{(n\log n)}$, $\Bar{\mathfrak{c}}_{\gamma}=\sqrtBis{(c_{2})}$ and $\VectCoord{\Tilde{Y}}{\gamma}=B$ is a standard Brownian motion and if $\gamma\in(1,2)$, then $\gamma_n=n^{1/\gamma}$, $\Bar{\mathfrak{c}}_{\gamma}=(\mathfrak{c}_{\gamma}|\Gamma(1-\gamma)|)^{1/\gamma}$ and $\VectCoord{\Tilde{Y}}{\gamma}=\VectCoord{Y}{\gamma}$ is a Lévy process with no positive jump such that for any $\lambda\geq 0$, $\mathtt{E}[e^{\lambda \VectCoord{Y}{\gamma}_t}]=e^{t\lambda^{\gamma}}$. Note that $\VectCoord{Y}{\gamma}$ has no positive jump because $\mathfrak{R}^1\geq 0$ which implies that the law of $\mathfrak{R}^1-1$ is supported on $\N\cup\{-1\}$. From now, without loss of generality, we assume that $\lim_{m\to\infty}m^{\gamma}\Pbis(\mathfrak{R}^1>m)=\mathfrak{c}_{\gamma}\in(0,\infty)$ exists for some $\gamma\in(1,2]$. Recall that one of our purposes is to prove the joint convergence in law, under $\Pbis$, for the Skorokhod product topology on $D([0,\infty))\times D([0,\infty))$ of 
$((\frac{1}{n}F_{\lfloor\alpha\gamma_n\rfloor};\; \alpha\geq 0),(\frac{1}{\gamma_n}\Bar{F}({\lfloor tn\rfloor});\; t\geq 0))_{n\in\N^*}$. For that, it is enough to prove the convergence in law for the Skorokhod topology on $D([0,\infty))$ of each component since the joint finite-dimensional convergence will follows immediately. 

\vspace{0.3cm}

\noindent\textbf{Convergence of $((\frac{1}{n}F_{\lfloor\alpha\gamma_n\rfloor};\; \alpha\geq 0))_{n\in\N^*}$}

\vspace{0.2cm}

\noindent Thanks to \eqref{ConvLuka2}, we obtain that for all $\alpha\geq 0$, $(\frac{1}{n}\#\mathfrak{F}^1_{\lfloor\alpha\gamma_n\rfloor})_{n\in\N^*}$ converges in law to $(\tau_{\alpha}(\Bar{\mathfrak{c}}_{\gamma}\VectCoord{\Tilde{Y}}{\gamma});\; \alpha\geq 0)$, having the same law as the process $(\tau_{\alpha}(\VectCoord{\Tilde{Y}}{\gamma})/(\Bar{\mathfrak{c}}_{\gamma})^{\gamma\land 2};\; \alpha\geq 0)$ thanks to the scaling property, see Remark \ref{RemTA}. We then deduce the convergence of $((\frac{1}{n}\#\mathfrak{F}^1_{\lfloor\alpha\gamma_n\rfloor};\; \alpha\geq 0))_{n\in\N^*}$. Indeed, for any $p'\in\N$, by the strong Markov property applied to $V^1$ at $\#\mathfrak{F}^1_{p'}$, we have that $(\#\mathfrak{F}^1_{p+p'}-\#\mathfrak{F}^1_{p'})_{p\in\N}$ is distributed as $(\#\mathfrak{F}^1_p)_{p\in\N}$ and is independent of $(\#\mathfrak{F}^1_p)_{p\leq p'}$. Moreover, using that $\#\mathfrak{F}^1_p=\inf\{k\in\N^*;\; -V^1_k\geq p\}$, we have, for any $\Bar{p}\in\N^*$, that $\sup_{p,p'\in\N^*;\; p\leq\Bar{p}}(\#\mathfrak{F}^1_{p+p'}-\#\mathfrak{F}^1_{p'})$ is smaller than a random variable distributed as $\#\mathfrak{F}^1_{\Bar{p}}$. Having said that, we deduce the finite-dimensional convergence of the sequence of processes. We also deduce that, for any $\varepsilon,\varepsilon_1>0$, $n\in\N^*$ and $\gamma_n>1/\varepsilon$ 
\begin{align*}
    \Pbis\Big(\sup_{\alpha,\beta\geq 0;\; |\alpha-\beta|<\varepsilon}\big|\#\mathfrak{F}^1_{\lfloor\gamma_n\alpha\rfloor}-\#\mathfrak{F}^1_{\lfloor\gamma_n\beta\rfloor}\big|\geq n\varepsilon_1\Big)\leq\Pbis\big(\#\mathfrak{F}^1_{\lfloor 2\gamma_n\varepsilon\rfloor}\geq n\varepsilon_1\big), 
\end{align*}
so we obtain that $\lim_{\varepsilon\to0}\limsup_{n\to\infty}\Pbis(\sup_{\alpha,\beta\geq 0;\; |\alpha-\beta|<\varepsilon}|\#\mathfrak{F}^1_{\lfloor\gamma_n\alpha\rfloor}-\#\mathfrak{F}^1_{\lfloor\gamma_n\beta\rfloor}|\geq n\varepsilon_1)$ is smaller than $\lim_{\varepsilon\to0}\mathtt{P}(\tau_{2/\Bar{\mathfrak{c}}_{\gamma}}(\VectCoord{\Tilde{Y}}{\gamma})\geq\varepsilon_1/\varepsilon^{\kappa\land 2})=0$. It is easy to see that $(\sup_{\alpha\in[0,M]}\#\mathfrak{F}^1_{\lfloor\gamma_n\alpha\rfloor}/n)_{n\in\N^*}$ is tight in $\R^+$ and finally yields the convergence of $((\frac{1}{n}\#\mathfrak{F}^1_{\lfloor\alpha\gamma_n\rfloor};\; \alpha\geq 0))_{n\in\N^*}$ thanks to \cite{BillingsleyBis} (Theorem 16.8 and equation (12.7)). \\
We are left to prove that the sequence of processes $((\mathfrak{D}(\lfloor nt\rfloor)/n;t\geq 0))_{n\in\N^*}$ goes to $(2\nu t; t\geq 0)$ in law, under $\Pbis$, for the Skorokhod topology on $D([0,\infty))$, where we recall that $\mathfrak{D}(k)$ is defined in \eqref{ForestType}. By Lemma \ref{Fact1} \ref{Fact1.1}, $\Ebis[\mathfrak{R}(j)]=\Ebis[\mathfrak{R}]=2\nu$ so the law of large numbers gives the finite-dimensional convergence. Moreover, since, for any $n\in\N^*$, the function $t\mapsto\mathfrak{D}(\lfloor nt\rfloor)/n$ is non-decreasing and the function $t\mapsto 2\nu t$ is continuous, Dini's theorem yields, for all $\varepsilon_1,M>0$, $\lim_{n\to\infty}\Pbis(\sup_{t\in[0,M]}|\frac{1}{n}\mathfrak{D}(\lfloor nt\rfloor)-2\nu t|>\varepsilon_1)=0$ thus giving the tightness of the sequence $(\sup_{t\in[0,M]}\mathfrak{D}(\lfloor nt\rfloor)/n)_{n\in\N^*}$ and yields 
\begin{align*}
    \lim_{\varepsilon\to 0}\limsup_{n\to\infty}\Pbis(\sup_{t,s\in[0,M];\; |t-s|<\varepsilon}|\mathfrak{D}(\lfloor nt\rfloor)-\mathfrak{D}(\lfloor ns\rfloor)|>n\varepsilon_1)=0,
\end{align*}
which gives the tightness of $((\mathfrak{D}(\lfloor nt\rfloor)/n;t\geq 0))_{n\in\N^*}$. \\
The convergence of $((\frac{1}{n}F_{\lfloor\alpha\gamma_n\rfloor};\; \alpha\geq 0))_{n\in\N^*}$ finally follows from \eqref{ForestType} together with the Skorokhod's representation theorem and using that $\gamma_n=o(n)$.

\vspace{0.3cm}

\noindent\textbf{Convergence of $(\frac{1}{\gamma_n}\Bar{F}({\lfloor tn\rfloor});\; t\geq 0))_{n\in\N^*}$}

\vspace{0.2cm}

\noindent Note that, thanks to \eqref{EncadreSupForest}, taking $\mathfrak{g}=n^{\theta}$ with $\theta\in(1/(\kappa\land 2),1)$, it is enough to prove the convergence of $((\max_{1\leq k\leq\Bar{\mathfrak{D}}(\lfloor nt\rfloor)}(-V^1_k)/\gamma_n;\; t\geq 0))_{n\in\N^*}$. Observe that $((\max_{1\leq k\leq\lfloor nt\rfloor}(-V^1_k)/\gamma_n;\; t\geq 0))_{n\in\N^*}$ converges in law to $((\Bar{\mathfrak{c}}_{\gamma}S(t,\VectCoord{\Tilde{Y}}{\gamma});\; t\geq 0))_{n\in\N^*}$. Indeed, the finite-dimensional convergence follows immediately from \eqref{ConvLuka2} and for the tightness of the sequence of processes, note that for any $p,p',\Bar{p}\in\N^*$ such that $p\leq\Bar{p}$, $\max_{1\leq k\leq p'+p}(-V^1_k)-\max_{1\leq k\leq\Bar{p}}(-V^1_k)\leq\max_{1\leq k\leq p}(-V^1_{k+p'}+V^1_{p'})$ so by the Markov property, for all $\varepsilon,\varepsilon_1>0$
\begin{align*}
    \Pbis\Big(\sup_{t,s\geq 0;\; |t-s|<\varepsilon}\Big|\max_{1\leq k\leq\lfloor nt\rfloor}(-V^1_k)-\max_{1\leq k\leq\lfloor ns\rfloor}(-V^1_k)\Big|\geq \gamma_n\varepsilon_1\Big)\leq\Pbis\Big(\max_{1\leq k\leq\lfloor2n\varepsilon\rfloor}(-V^1_k)\geq\gamma_n\varepsilon_1\Big), 
\end{align*}
so we obtain that $\lim_{\varepsilon\to0}\limsup_{n\to\infty}\Pbis(\sup_{t,s\geq 0;\; |t-s|<\varepsilon}|\max_{1\leq k\leq\lfloor nt\rfloor}(-V^1_k)-\max_{1\leq k\leq\lfloor ns\rfloor}(-V^1_k)|\geq \gamma_n\varepsilon_1)$ is smaller that $\lim_{\varepsilon\to0}\mathtt{P}(S(2(\Bar{\mathfrak{c}}_{\gamma})^{\gamma},\VectCoord{\Tilde{Y}}{\gamma})>\varepsilon_1/\varepsilon^{1/\gamma})=0$. \\
We are left that to prove that $((\Bar{\mathfrak{D}}(\lfloor nt\rfloor)/n;\; t\geq 0))_{n\in\N^*}$ converges in law to $(t/(2\nu);\; t\geq 0)$, where we recall the definition of $\Bar{\mathfrak{D}}$ in \eqref{InvDQ}. This convergence also follows from Dini's theorem, using the convergence of $(\mathfrak{D}(n)/n)_{n\in\N^*}$ to $2\nu$, the fact that $t\mapsto\Bar{\mathfrak{D}}(\lfloor nt\rfloor)$ is non-decreasing for all $n\in\N^*$ and that $t\mapsto t/(2\nu)$ is continuous. We finally deduce the convergence of $(\Bar{F}({\lfloor tn\rfloor})/\gamma_n;\; t\geq 0))_{n\in\N^*}$ using the Skorokhod's representation and the scaling property of $\VectCoord{Y}{\gamma}$.

\vspace{0.3cm}

\noindent The last step consists on showing that for any $M,\varepsilon_1>0$
\begin{align*}
    \lim_{n\to\infty}\Pbis\Big(\sup_{1\leq p\leq\lfloor M\gamma_n\rfloor}\Big|\sum_{x\in\mathcal{F}_p}1-\frac{\Tilde{\nu}}{2\nu}F_p\Big|>\varepsilon_1 n\Big)=0.
\end{align*}
We already know that for any $p\in\N^*$, $F_p=p+\mathfrak{D}(\#\mathfrak{F}^1_p)$, see \eqref{ForestType} and also that for any $M',\varepsilon_2>0$, $\lim_{n\to\infty}\Pbis(\sup_{1\leq k\leq\lfloor M'n\rfloor}|2\nu k-\mathfrak{D}(k)|>n\varepsilon_2)=0$. Similarly, one can see that $\sum_{x\in\mathcal{F}_p}1=\sum_{x\in\Tilde{\mathcal{F}}_p}1=p+\Tilde{\mathfrak{D}}(\#\mathfrak{F}^1_p)$ where $\Tilde{\mathcal{F}}_p$ is the forest composed of the first $p$ trees of the forest $(\Tilde{\mathcal{T}}_i)_{i\in\N^*}$, see \ref{ProofProp}, and $\Tilde{\mathfrak{D}}(k):=\sum_{j=0}^{k-1}\Tilde{R}(j)$, with $\Tilde{R}(j)$ the number of children of the $(j+1)$-th vertex (for the depth first order) of the forest $(\Tilde{\mathcal{T}}_i)_{i\in\N^*}$. Moreover, by definition, $\Ebis[\Tilde{R}(j)]=\Tilde{\nu}$ for all $j\in\N$ and we have $\lim_{n\to\infty}\Pbis(\sup_{1\leq k\leq\lfloor M'n\rfloor}|\Tilde{\mathfrak{D}}(k)-\Tilde{\nu}k|>n\varepsilon_2)=0$ thus giving 
\begin{align*}
    \lim_{n\to\infty}\Pbis\Big(\sup_{1\leq k\leq\lfloor M'n\rfloor}\Big|\Tilde{\mathfrak{D}}(k)-\frac{\Tilde{\nu}}{2\nu}\mathfrak{D}(k)\Big|>n\varepsilon_2\Big)=0.
\end{align*}
Hence, decomposing according to the value of $\sup_{1\leq p\leq\lfloor M\gamma_n\rfloor}\#\mathfrak{F}^1_p=\#\mathfrak{F}^1_{\lfloor M\gamma_n\rfloor}$
\begin{align*}
    &\Pbis\Big(\sup_{1\leq p\leq\lfloor M\gamma_n\rfloor}\Big|\Tilde{\mathfrak{D}}(\#\mathfrak{F}^1_p)-\frac{\Tilde{\nu}}{2\nu}\mathfrak{D}(\#\mathfrak{F}^1_p)\Big|>\varepsilon_1 n\Big) \\ & \leq\Pbis\Big(\sup_{1\leq p\leq\lfloor M\gamma_n\rfloor}\Big|\Tilde{\mathfrak{D}}(\#\mathfrak{F}^1_p)-\frac{\Tilde{\nu}}{2\nu}\mathfrak{D}(\#\mathfrak{F}^1_p)\Big|>\varepsilon_1 n,\; \#\mathfrak{F}^1_{\lfloor M\gamma_n\rfloor}\leq\lfloor M'n\rfloor\Big)+\Pbis\big(\#\mathfrak{F}^1_{\lfloor M\gamma_n\rfloor}>\lfloor M'n\rfloor\big) \\ & \leq \Pbis\Big(\sup_{1\leq k\leq\lfloor M'n\rfloor}\Big|\Tilde{\mathfrak{D}}(k)-\frac{\Tilde{\nu}}{2\nu}\mathfrak{D}(k)\Big|>n\varepsilon_2\Big)+\Pbis\big(\#\mathfrak{F}^1_{\lfloor M\gamma_n\rfloor}>\lfloor M'n\rfloor\big),
\end{align*}
so the probability $\Pbis(\sup_{1\leq p\leq\lfloor M\gamma_n\rfloor}|\Tilde{\mathfrak{D}}(\#\mathfrak{F}^1_p)-\frac{\Tilde{\nu}}{2\nu}\mathfrak{D}(\#\mathfrak{F}^1_p)|>\varepsilon_1 n)$ goes to $0$ as $n\to\infty$ taking $M'$ large enough and this ends the proof.
\end{proof}

\vspace{0.5cm}

\noindent\begin{merci}
    This work is funded by Hua Loo-Keng Center for Mathematical Sciences (AMSS, Chinese Academy of Sciences) and supported by the National Natural Science Foundation of China (No. 12288201).
\end{merci}

\bibliographystyle{alpha}
\bibliography{thbiblio}

\end{document}